\documentclass[11pt]{article}
\usepackage{amsfonts,amsmath,amsthm,amssymb,stmaryrd}
\usepackage{authblk}
\usepackage{wasysym, url}
\usepackage[utf8]{inputenc}
\usepackage{xspace, xcolor, soul}

\author{Neil Dobbs\footnote{
    School of Maths and Stats, University College, Dublin, Ireland}}

\author{Jacek Graczyk\footnote{
    Univ. Paris-Sud, Lab. de Math\'ematiques, 91405 Orsay, France}}

\author{Nicolae Mihalache\footnote{
    Univ Paris Est Creteil, CNRS, LAMA, F-94010 Creteil, France \emph{and}\\
 Univ Gustave Eiffel, LAMA, F-77447 Marne-la-Vall\'ee, France}}
\affil{}

\date{}

\newtheorem{lem}{Lemma}[section]

\newtheorem{theo}{Theorem}
\newtheorem{coro}[lem]{Corollary}
\newtheorem{prop}[lem]{Proposition}

\newtheorem{defi}[lem]{Definition}
\newcommand{\cN}{{\mathcal N}}

\newcommand{\cU}{{\mathcal U}}
\newcommand{\A}{{\mathcal A}}
\newcommand{\M}{{\mathcal M}}

\newcommand{\calD}{{\mathcal D}}
\newcommand{\J}{{\mathcal J}}
\newcommand{\G}{{\mathcal G}}
\newcommand{\de}{{\bf \delta}}

\newtheorem{fact}[lem]{Fact}

\def\ss{\subset}
\def\se{\subseteq}
\def\sm{\setminus}
\def\pa{\partial}
\def\ol{\overline}

\font\mathfonta=msam10 at 11pt
\font\mathfontb=msbm10 at 11pt
\def\Bbb#1{\mbox{\mathfontb #1}}
\def\lesssim{\mbox{\mathfonta.}}

\def\grtsim{\mbox{\mathfonta\&}}

\def\diam{{\mathrm{diam}\,}}
\def\dist#1{{{\mathrm{dist}}\br{#1}}}

\def\eps{\epsilon}
\def\del{\Delta}

\def\be{\beta}
\def\la{\lambda}
\def\om{\omega}
\def\Om{\Omega}
\def\al{\alpha}
\def\th{\theta}

\def\dd{\partial}

\def\J{{\mathcal{J}}}
\def\H{{\mathcal{H}^1}}
\def\vp{{\varphi}}
\def\tvp{\tilde{\varphi}}

\def\C{\Bbb C}
\def\N{\Bbb N}
\def\Z{\Bbb Z}
\def\D{\Bbb D}

\def\R{{\Bbb{R}}}
\def\br#1{\left(#1\right)}
\def\brs#1{\left\{#1\right\}}
\def\ii#1{\llbracket #1 \rrbracket}

\def\HD{{\mathrm{dim_H}}}
\def\dconf{{\delta_{\mathrm {conf}}}}

\def\ra{{\rightarrow}}
\def\fr{\frac}
\def\itemm#1{ \item{\makebox[0.3in][l]{#1}}}

\def\requ#1{(\ref{equ:#1})}

\def\cblue{\color{blue}}
\def\cred{\color{blue}}
\def\cblue{}
\def\cred{}

\begin{document}
\title{Hausdorff Dimension of Julia sets in the logistic family}

\maketitle

\begin{abstract}  
A closed interval and circle are the only smooth Julia sets in polynomial dynamics. D.\ Ruelle proved that the Hausdorff
dimension of unicritical Julia sets close to the circle depends analytically on the parameter. Near the tip
of the Mandelbrot set $\M$, the Hausdorff dimension is generally discontinuous.  Answering a question of J-C.\ Yoccoz in the conformal setting, we observe that the  Hausdorff dimension of quadratic Julia sets depends continuously on $c$ and find explicit bounds at the tip of $\M$ for most real parameters in the the sense of $1$-dimensional Lebesgue measure.
\end{abstract}

\section{Introduction}
\paragraph{H\'enon attractors.}
There is a natural  connection between quadratic maps $f_c(z) = z^2 +c$, $c \geq -2$, and H\'enon maps
$$T_{a,b}(x,y)=(x^2+a+y,bx), \;\;\;$$ 
$a\geq-2$, $b\geq0$.  
J-C.\ Yoccoz posed the following  question in the strongly dissipative planar setting corresponding to $c\to-2^+$,~\cite{ber,yoc1}. Is it true that the Hausdorff dimension of the H\'enon attractor $X_{a,b}$ (the closure of the unstable manifold) of $T_{a,b}$, 
tends  to $1$ as $(a,b) \to (-2,0)$, for $(a,b)$ in a set $A$ having $(-2,0)$ as a Lebesgue density point
in the quadrant $(-2,+\infty)\times (0,\infty)$? If yes, what is the asymptotic behaviour of $(a,b)\mapsto \HD(X_{a,b})$ when
$A \ni (a,b)\rightarrow (-2,0)$?
 
It is known, by~\cite{Becaa, BL}, that $A$ can be chosen so that $T_{a,b}$ has a unique SRB invariant measure $\mu = \mu_{a,b}$. By the dimension formula~\cite{LS}, 
$$ \HD(\mu)=h_\mu\cdot \left[ \frac{1}{\lambda_1}-\frac{1}{\lambda_2}\right], $$
$\lambda_1$ and $\lambda_2 $ are the Lyapunov exponents of $\mu$, $\HD(\mu)$ is the infimum of the Hausdorff dimension of sets of full  $\mu$-measure, and
$h_\mu$ is the metric entropy. By Pesin's formula
$h_\mu$ is equal to $\lambda_1$. Clearly,
$\lambda_1 +\lambda_2=\log b,$ and $\lambda_1\sim \log 2$
for $(a,b)$ close to $(-2,0)$,  this means that
$$\HD(X_{a,b})\geq \HD(\mu)=1+\frac{\lambda_1}{|\lambda_2|}
\geq 1+\frac{C}{|\log b|}$$
for an absolute constant $C>0$, giving a weak estimate for the lower bound.

\paragraph{Main result.}
The following theorem answers  the logistic family counterpart to J-C.\ Yoccoz' questions, obtaining strong bounds on the asymptotic behaviour as $c$ approaches the one-sided density point $-2$ for a positive measure set of real Collet-Eckmann parameters ($c \in [-2,0]$ for which $\liminf_{n\to \infty} \frac1n \log |(f_c^n)'(c)| >0$). 
\begin{theo} \label{theo:main}
    Let $\HD(\J_c)$ denote the Hausdorff dimension of the Julia set of the quadratic map $z \mapsto z^2 + c \in \C$.
    For some constant $C>1$, 
    $$
    1 + C^{-1} \sqrt{c + 2} \leq \HD(\J_c) \leq 1 + C \cred{ |\log (c+2)| } \sqrt{c+2}$$
    for all $c$ from a subset of real Collet-Eckmann parameters  with 
    the point $-2$ as a one-sided Lebesgue density point.
\end{theo}
Both the upper and the lower bounds are new.
The asymptotic lower
bound  for the Hausdorff dimension holds for all $c>-2$, not just a positive measure set, see Theorem~\ref{theo:hausBot}.  

One of our objectives, in view of Yoccoz' problem, was to  develop methods for the upper estimates of the Hausdorff dimension which 
 generalize to non-analytic and higher dimensional systems  where, naturally,  technicalities are more challenging. 
The techniques to prove the lower bound are more anchored in
conformal dynamics. A more detailed discussion of our approach can be found in Section~\ref{sec:met}.

\paragraph{Unimodal maps, logistic family and Mandelbrot set.}
A differentiable map $f$ of the interval $I$ is called unimodal if it has exactly one critical point,    $f'(\alpha)=0$, and maps the boundary of $I$ into itself.  

The logistic family  $x\mapsto ax(1-x)$, $a\in(0,4]$ is both the simplest and the most known model of non-linear dynamics. The family is often studied in the form 
$z^2+c$, $c\in [-2,1/4]$ which is convenient in the complex setting $z,c\in \C$. For a fixed
$c\in \C$, the Julia set $\J_c$ is defined as  
the smallest \cred{non-trivial} totally invariant compact set, $\J_c=f^{-1}_c(\J_c),$ and is the locus of chaotic dynamics.
In the real case, 
\emph{density of hyperbolicity} implies that the dynamics on $\J_c\cap \R$ is structurally unstable whenever the parameter $c \geq -2$ is not hyperbolic. 

The logistic family in the complex parametrization is embedded inside the Mandelbrot set $\M$. Namely,
$[-2,1/4]=\M\cap \R$, where
$$\M=\{c\in \C: \forall_{n\geq 0}~|f_c^n(c)|\leq 2\}.$$
The Mandelbrot set is one of the most studied irregular fractals in science and a vast literature exists in relation to its various remarkable properties, see for example~\cite{caga}.

Even in the simplest quadratic case $f_c(z)=z^2+c$, the Hausdorff dimension of Julia sets $\HD(\J_c)$ is notoriously difficult to estimate. The exact values of  $\HD(\J_c)$  are known  only for $c=0$ and $-2$.
M. Shishikura proved in ~\cite{shish} that for a topologically generic set in $\partial{\M}$, $\HD(\J_c)=2$, while the results of~\cite{harm, smir, nonhyp} yield that $\HD(\J_c)<2$ for almost all parameters $c\in {\M}$ with respect to the harmonic measure. Another consequence of~\cite{shish} is that 
the dimension function $c\in \partial {\M}\mapsto \HD(\J_c)$ is discontinuous at every $c\in \partial {\M}$ such that $\HD(J_c)<2$.

At parabolic parameters such as $c=1/4$, the dimension function may be discontinuous, as observed in \cite{parab}, 
due to parabolic implosion~\cite{D,etiuda}. This direction of research was advanced further in \cite{Jaksz, HavZin} in relation to A. Douady's
longstanding project to understand oscillations of $\HD(\J_c)$ 
near a parabolic bifurcation.

Certain continuity properties of $c\in {\M} \mapsto \HD(J_c)$ were obtained in~\cite{mcmul1,mcmul2} and~in~\cite{nonhyp}.

Even restricting to $c\in \partial \M \cap \R$, the dimension function is discontinuous at $-2$, as detailed in Proposition~\ref{prop:discon}. 
We shall study the asymptotic behaviour at $-2$ of  a large set of real parameters with respect to Lebesgue measure.

\subsection{Real parameters}
\paragraph{Typical parameters.}
By the work of Benedicks and Carleson,~\cite[Theorem~1]{Becaa} and the summary on~\cite[page 87]{Becaa}, there exist positive constants $\Omega' \in (0,1)$ and $\Omega \in (0,\log 2)$ such that the set of non-hyperbolic parameters
\begin{equation}\label{equ:BC}
    \hat \A := \{c \in {\M} \cap \R \,:\, \forall_{n \geq 1}\; |(f_c^n)'(c)| \geq \Omega' e^{\Omega n}\; \}
\end{equation} 
has full one-sided Lebesgue density at $-2$ (see also Collet and Eckmann \cite{CE:Abundance}). 
Fix such a set
$\hat \A$  of uniform \emph{Collet-Eckmann} parameters.
This was refined further by J.-C. Yoccoz in \cite{YocJak} to obtain an explicit bound on $\Omega$ 
in terms of $c+2$. Our Proposition~\ref{prop:realintro} will show that even if one took
 $\Omega'=1$ and $\Omega$ arbitrarily close to  $\log 2$, $\hat \A$ would have positive measure with $-2$ as a one-sided density point. 

We shall be interested in the dynamics as we appoach the parameter $-2$. For $c > -2$, we shall call $c$, and the associated dynamics, \emph{typical}, if $c \in \hat \A$.

\paragraph{Non-hyperbolic dynamics.}
By the work of M.\ Jakobson~\cite{jak}, for $c>-2$, a typical parameter exhibits chaotic dynamics  on $[c,f_c(c)]$ and the asymptotic distribution of almost every orbit is given by an invariant measure absolutely continuous with respect to
$1$-dimensional Lebesgue measure. A complex counterpart
of M. Jakobson theorem follows from \cite{nonhyp,Becaa}: for  typical $c > -2$, there is a continuous conformal measure on $\J_c$, with respect to which almost every orbit distributes according to the unique invariant probabilistic measure, equivalent to the conformal one. 

\paragraph{Strong phase transition at the tip of the Mandelbrot set.}
If $c<-2$ then $f_c$ is  hyperbolic and the half-line $(-\infty, -2]$ is a hyperbolic geodesic in $\hat{\C}\setminus {\M}$ landing at $-2$. The estimates of~\cite{mcmul1} and~\cite{nonhyp} imply that   
$$\lim_{c\rightarrow -2^{-}}\HD(J_c)=\HD(J_{-2})\,.$$
Applying real methods, a  precise H\"older  estimate was obtained in~\cite{jiang} for $c\in (-\infty, -2]$,
$$
C^{-1}\sqrt{|c+2|} \leq
1-\HD(\J_c)
\leq C \sqrt{|c+2|}
$$
where $C>1$ is a universal constant.
 The result is further discussed in Section~\ref{sec:jiang}.

By~\cite{ruelle}, the dimension function $c\in \M\cap\R\mapsto \HD(\J_c)$ is real analytic for $c<-2$. However, 
at the unfolding parameter $c=-2$ a strong bifurcation takes place, the relation $c\mapsto\HD(\J_c)$ enters a discontinuous regime and its oscillations at the right-hand side of $-2$ are extreme. 
\begin{prop}\label{prop:discon}
The dimension  function $c\in \M\cap \R \mapsto \HD(\J_c)$ is discontinuous at $-2$. Moreover,
$$\limsup_{c\rightarrow -2} \HD(J_{c})=\sup_{c\in {\M}\cap \R} \HD(J_c)>1=\liminf_{c\rightarrow -2^+} \HD(J_{c})\,.$$
\end{prop}

\begin{proof}
 
G.\ Levin and M. Zinsmeister,~\cite{lezi}, showed that  
    $\HD(\J_c) > \frac{4}{3}$ for a (non-connected) open set of real parameters with $-2$ in its boundary, showing discontinuity. 

Let $\hat c$ be a non-zero parameter in $(-2,1/4]$.
    Given $c \in [-2,-1]$, denote by $p(c)$ the orientation-preserving and $q(c)$ the orientation-reversing fixed points of $f_c$, so $0 < -q(c) < p(c)$.
    For each $n\geq 5$, denote by $I_n$ the parameter set of $c \in \R$ for which 
    $$
    f_c(0) = c < q(c) <  f_c^n(0) < -q(c) < f_c^{n-1}(0) <  f_c^{j+1}(0) < f_c^j(0) 
    $$
    for all $j = 2, 3, \ldots, n-3$. 
    $I_n$ contains an interval of parameters $J_n$  for which $f^n_c$ has a symmetric closed forward-invariant subinterval $W_n(c) \subset (q(c), -q(c))$. The set of maps $\{f^n_c\}_{c\in J_n}$ restricted to $W_n(c)$ forms a full unimodal family, see \cite[Section~II.4]{demelovanstrien}, so for some $c_n \in J_n$, 
    $f_{c_n}^n$ restricted to $W_n(c)$ is conjugate to $f_{\hat c}$ restricted to $[-p(\hat c), p(\hat c)]$. 
    We fix such a sequence $(c_n)_{n\geq 5}$. 
    Standard arguments (reprised later in the paper) show that $c_n+2 \sim 4^{-n} \sim p(c_n) + c_n$ and that there is a topological disc $V'_n$ containing $c_n$ of diameter $ \sim 4^{-n}$ mapped univalently by $f^{n-1}_{c_n}$ onto 
    $U_n = \D_{[q(c_n),-q(c_n)]}$, the disc whose diameter is the line segment $[q(c_n),-q(c_n)]$. 
    Then $V_n := f^{-1}_{c_n}(V'_n)$ is a topological disc of diameter $\sim 2^{-n}$ mapped by $f^n_{c_n}$ as a two-to-one branched cover onto $U_n$. 

    The orbit of $0$ under $f_{c_n}^n$ lies in $V_n$ and $f^n_{c_n} : V_n \to U_n$ is a \emph{quadratic-like} map and conjugate as such to $f_{\hat c}$. 
    See \cite{douhub} for properties of polynomial-like mappings.  Observe that, as $c_n \to -2$ when $n \to \infty$, $U_n \to \D_{[-1,1]}$. As $0 \in V_n$ and $\diam(V_n) \sim 2^{-n}$, the modulus of $(U_n\setminus V_n)$ grows $ \sim n$. 
 There exist  $K_n$-quasiconformal maps $h_n:\C\ra \C$ 
such that for every $z\in V_n$,
$$h_n\circ f^{n}_{c_n}=f_c\circ h_n\,$$ 
    and $\lim_{n\rightarrow \infty}K_n=1$, see \cite[Section 11]{sullquasi}. In particular, 
    $$\liminf_{n\rightarrow \infty} \HD(J_{c_n})\geq \HD(J_{\hat c}).$$
    By A.\ Zdunik's theorem~\cite{z}, 
    $\HD(J_{\hat c}) >1,$
from which the
limsup
estimate of Proposition~\ref{prop:discon} follows.
The liminf estimate is a direct consequence of Theorem~\ref{theo:start} or~\cite{al}.
\end{proof}

To the best of our knowledge, the value of 
\begin{equation}\label{equ:lz}
\sup_{c\in {\M} \cap \R} \HD(J_c)
\end{equation}
remains unknown. 
It is unknown whether
the supremum 
is strictly less than $2$.

In~\cite{al}, 
examples were given of  sequences of Julia sets $J_c$, $c\in {\M}\cap \R$, of infinitely renormalizable quadratic polynomials $f_c$ for which
  $\HD(J_c)\to 1 $ when $c$ tends to $-2$. 

Theorem~\ref{theo:main}, which we discuss in Section~\ref{sec:met}, estimates the dimension for $c$ near $-2$ for a \emph{large} set of parameters. A weaker, preliminary result,
Theorem~\ref{theo:start}, asserts that continuity of the dimension function $c\mapsto \HD(\J_c)$ holds for typical parameters $c>-2$.
\begin{theo}\label{theo:start} The dimension function
$$c\in \hat {\A} \mapsto \HD(\J_c)$$
is continuous at every parameter $c\in \hat {\A}$.
\end{theo}
\begin{proof}
This theorem follows from the fact that every map $f_c$,
$c\in \hat {\A}$ is uniformly summable in the sense
of~\cite{nonhyp} and by invoking Theorem~{\cred 12} of~\cite{nonhyp}.
\end{proof}
In particular, Theorem~\ref{theo:start} answers the first part of Yoccoz' question in the quadratic setting and is a starting point for the research presented in the current paper.
A more general version of Theorem~\ref{theo:start} formally follows from~\cite{bengra}, namely the limit
exists in the sense of Theorem~\ref{theo:start}
along any $C^2$-curve ending at $-2$.



\subsection{Unfolding families}
It is known, by~\cite{z}, that $\HD(\J_c)>1$ for all $c\in \M\setminus\{0,-2\}$.
At $c=0$ or $-2$, the Julia set is respectively the unit circle or the segment $[-2,2]$.
We can take a one-parameter family traversing either $-2$
or $0$ to understand how an analytic set unfolds into
a fractal. A natural quantity to capture the unfolding is $\HD(\J_c)$.

\paragraph{Structurally stable case.}
D.~Ruelle proved in \cite{ruelle} that the function
$c\in {\M} \mapsto \HD(\J_c)$  is real analytic at every hyperbolic $c\in \C$ and its asymptotic behaviour at $c=0$ is given by
\begin{equation}\label{equ:ruelle}
\HD(\J_c)=1+\frac{|c|^2}{4\log 2} + O(|c|^3)\,.
\end{equation}

The methods used in~\cite{ruelle} are an early and innovative application of thermodynamical formalism.

\paragraph{Logistic family.}
Our objective is  to understand how  non-linear attractors unfold within one-parameter families of bifurcating maps.
The logistic family at $-2$ is a prototype example of bifurcating dynamics. Our research shows that one of the most important averaging parameters of the system, the Hausdorff dimension $\HD(\J_c)$, for typical parameters $c$,
shows almost H\"older dependence on $c$, Theorem~\ref{theo:hausTop} and \ref{theo:hausBot}.
Additionally, the lower bound of $\HD(\J_c)$ is valid for all parameters $c\in \R$ from a vicinity of $-2$. 


\subsection{Methods and statement of the results}\label{sec:met}
Uniform and explicit estimates of the dimension 
function $c\mapsto \HD(\J_c)$ in bifurcating families are a major technical challenge. Lower bounds on the Hausdorff dimension are usually more difficult to obtain. 
We prove, in Section~\ref{sec:rep}, a uniform lower bound for all parameters by constructing an induced Cantor repeller and applying Sinai-Ruelle-Bowen methods of thermodynamical formalism. A careful combinatorial and probabilistic argument yields the final result, Theorem~\ref{theo:hausBot}.

The upper bound is, strangely, more involved.
The Hausdorff dimension of Julia sets is strongly discontinuous in the upward direction due to ongoing bifurcations. It is enough that a Julia set grows locally in one scale or another and a similar growth is inherited in most scales and points due to invariance and dynamics. This phenomena is well-known in complex dynamics
and leads to discontinuities in the dimension function $c\mapsto \HD(\J_c)$~\cite{parab}. 
The upper estimates rely on the technique of the Poincar\'e series,~\cite{nonhyp}, statistical properties of conformal densities, 
the theory of $\be$-numbers from~\cite{gjm}, and the parameter exclusion  constructions of~\cite{Becaa, YocJak}.
\paragraph{Theory of $\beta$-numbers and dynamics.}
P.\ Jones~\cite{jo1990} introduced the technique of $\beta$-numbers to study geometry of planar sets from the viewpoint of  the theory of $L^2$ functions. One of the outcomes was quantification of the intuition that sets which wiggle in most scales must be metrically large. 

\begin{defi}\label{defi:beta}
Let $K \se \R^d$ with $d\geq 2$, $x\in K$ and $r>0$.   
We define $\beta_{K}(x,r)$ by
$$ \beta_{K}(x,r):=  \inf_{L} \sup_{z\in K\cap B(x,r)} \frac{\dist{z,L}}{r}\ ,$$
where the infimum is taken over all lines $L$ in $\R^d$.
\end{defi}
A bounded  set $K$ is called {\em uniformly wiggly} if 
$$\beta_{\infty}(K): = \inf_{x\in K}\inf_{r\leq \diam K} \beta(x,r)>0\;. $$
\cite[Theorem~1.1]{bijo2} states that if $K\subset \C$ is a continuum and $\beta_\infty(K)>0$
then $\HD(K)\geq 1+c\beta_{\infty}^{2}(K)$, where $c$ is a universal constant.

In dynamical systems one cannot expect that  generic systems be uniformly hyperbolic and that the geometry
of invariant fractals can be controlled at every scale, as is required in~\cite{bijo2}. However,  there are numerous results  showing that non-uniformly hyperbolic systems are typical in ambient  parameter spaces;
examples include the logistic family~\cite{jak, Beca}, rational maps~\cite{rees, aspenberg}, quadratic polynomials $z^2+c$ with $c\in \partial \M$~\cite{harm, smir},  H\'enon family~\cite{Becaa}.

To provide tools to study  attractors/repellers  of non-uniformly hyperbolic systems, 
 \cite[Theorem~1]{gjm} states that if a continuum $K\ss\C$ is the union of two subsets $K=W\cup E$,  $\H(E)<\infty$ and $\H(W)>0$, then
\begin{equation}\label{equ:HD}
\HD(K) \geq 1 + C \inf_{x\in W}\,\liminf_{r\ra 0}
\frac{\int_r^{\diam \!K}\be_K^2(x,t)\frac{dt}t}{-\log r}, 
\end{equation}
where $C$ is a universal constant and $\H$ is the $1$-Hausdorff measure. 
 If the infimum in $(\ref{equ:HD})$ is bigger than $\beta_0^2$ then we say that continuum $K$ is {\em mean wiggly} with the parameter $\beta_0\geq 0$. 
 
It follows from~\cite{gjm,przro} 
 that Julia sets of quadratic polynomials $\J_c$ are uniformly mean wiggly with $\beta_0$ comparable to $\sqrt{c+2}$, $c\in \hat {\A}$.
 This yields an ad hoc estimate for $c\in\hat  \A$,
 $$\HD(\J_c)\geq 1+C|c+2|$$
that is much weaker than the estimate given by
Theorem~\ref{theo:hausBot}. The explanation lies in the fact that $(\ref{equ:HD})$ is valid for all continua including self-similar curves of von Koch type where the length in a given scale $r$ in terms of $\beta:=\beta(x,r)$ is, by the Pythagorean theorem, at least a constant multiple of 
$$r\sqrt{1+\beta^2}\sim r(1+\beta^2/2).$$
The situation is very different for Julia sets $\J_c$, $c\in \hat \A$ which have a very particular property that they contain, at most points and scales $r>0$, a ``cross'' of length $\sim r(1+\beta)$, that corresponds to the lower estimate  
Theorem~\ref{theo:hausBot}.

Surprisingly, according to  \cite[Theorem~2]{gjm}, estimate $(\ref{equ:HD})$ can be almost inversed 
\begin{equation}\label{equ:HD1}
\HD(W) \leq 1 + C' \sup_{x\in W}\,\liminf_{r\ra 0}
\frac{\int_r^{\diam \!K}\be_K^2(x,t)\frac{dt}t}{-\log r}, 
\end{equation}
where $C'>0$ is a universal constant. If the supremum in  $(\ref{equ:HD1})$ is smaller than $\beta_0^2$ then $W$ is {\em almost flat} with  $\beta_0\geq 0$, see  Fact~\ref{thminvmean}. In the non-uniformly hyperbolic setting, of which $c\in \hat \A$ is a particular case, when $K$ is a Julia set, $W$ can be usefully chosen in $(\ref{equ:HD1})$ so that $\HD(K)=\HD(W)$. Indeed,  the unique  invariant conformal density has the property that it gives positive mass only to sets of the Hausdorff dimension $\HD(K)$,~\cite{nonhyp}. An additional advantage of such a choice is that every point from $W$ is typical with respect to the dynamics and the Birkhoff averages converge. This turns out to be crucial for efficient estimates  of the integral quantity
$$\liminf_{r\ra 0}
\frac{\int_r^{\diam \!K}\be_K^2(x,t)\frac{dt}t}{-\log r}$$
which is affected by the density of scales at $x$ where $\beta_K(x,r)$ is large. 





\paragraph{Passages to the large scale.} 
The frequency with which points go univalently to the large scale will play a r\^ole in our analysis. With this in mind, we introduce the following terminology. 

\begin{defi}\label{def:pullback}
Given $z \in \C, r>0$ and $f=f_c$, we denote by $$\cU(z,r)$$
the set of positive integers $n$ for which, for some neighbourhood $U$ of $z$, the map $f^n : U \to B(f^n(z), r)$ is biholomorphic. We call $U$ a \emph{level-$n$ univalent pullback} of $B(f^n(z), r)$  containing $z$, or to $z$. 
\end{defi}

More generally, if $W$ is connected and $V$ is a connected component of $f^{-n}(W)$, we call $V$ a \emph{level-$n$ pullback} of $W$. If $W$ is simply connected then, as $f^n$ is a polynomial, so is $V$. 

\paragraph{Passages of the critical value to the large scale.} 
In particular, the frequency with which the critical value goes univalently to the large scale has importance for us.
    \begin{prop}\label{prop:realintro}
        Given $\delta>0$, there exist  constants $\kappa, c_0 >-2$ {\cred and a set of parameters $\A \se \hat {\A} \cap [-2, c_0]$ which has $-2$ as a one-sided Lebesgue density point and the following properties.} For every $c \in \A$, $n\in \N$, there exists 
        $$n^*  \in \cU(c, 1/2) \cap \left[n, \max\left(n(1+\tau(c+2)), {\cred n - \log(c+2)}\right)\right], $$
where $\tau(\eps)=\exp(-\kappa\sqrt{-\log \eps})$.

For every $n$, 
$$ |(f_c^n)'(c)| \geq e^{n(\log 2 - \delta)}.$$
       
\end{prop}
{\cred 
This is essentially due to Yoccoz~\cite{YocJak} and we prove the proposition in Section~\ref{sec:yoc}. There is a set $\A_{\mathrm{Yoc}} $ of \emph{strongly regular} parameters defined by Yoccoz. We shall show that $\A$ can be taken in the form $\A_{\mathrm{Yoc}} \cap [-2, c_0]$.  
We fix such a set $\A$, for $\delta = 1/10$. 
Remark that $c_0 > -2$ may be further decreased and the conclusion of the proposition above remains valid.}

\paragraph{Conformal and absolutely continuous invariant measures.}
It is known~\cite{nonhyp} that every map $f_c$, $c\in \hat {\A}$,
has a unique probabilistic invariant absolutely continuous invariant measure $\sigma_c$ with respect to the unique \emph{geometric measure} $\nu_c$, $\nu_c(\J_c)=1$,
$$ \nu_c(f_c(U))=\int_{U}|f_c'(z)|^{\HD(\J_c)}~d\nu_c~ $$
for any Borel set $U$ on which $f$ is injective, 
see Definition~\ref{def:mes}.

\cred{
    We shall eventually apply Birkhoff's ergodic theorem; to do so, we require new, strong estimates on the distribution of $\sigma_c$ near $c$ (or equivalently, near $0$). 
The following two propositions are proven in Section~\ref{sec:stat}, see the \emph{diagonal estimate} on page~\pageref{prop:sigreg} and the \emph{general upper bound} on page~\pageref{page:epspf1}. 

Let $\sigma_c(r):=\sigma_c(B(c,r))$ and $\eps=c+2$.
\begin{prop}
Given $t_0 > 1$, there exists $C>1$ such that,
if $c \in  \A$ and $c+2 = \eps$ is sufficiently small, then for $r = a\, \eps$, $a \geq \eps^{t_0}$,
\begin{eqnarray}
\sigma_c(a\,\eps) 
	\leq C\,  a^{1/2} \, \sqrt{\eps} \;.
\label{equ:eps}
\end{eqnarray}
\end{prop}

For smaller scales there is a weaker estimate. 
\begin{prop}
There exists $C>1$ such that,
if $c \in  \A$ and $\eps$ is sufficiently small, then for $r = a\, \eps >0$,
\begin{eqnarray}
\sigma_c(a\,\eps) 
	\leq C\,  
     \max(a^{1/2}, a^{2/5}) \, \sqrt{\eps |\log\eps|} \;.
\label{equ:eps1}
\end{eqnarray}
\end{prop}

}

\paragraph{Upper estimates of $\HD(\J_c)$.}
Large scale geometric parameters for typical Julia sets $\J_c$, $c\in {\A}$, like global flatness~\cite{gjm},  as well as regularity of the invariant conformal density, enter into the upper bound of Theorem~\ref{theo:hausTop}. This surprising observation can be explained through the ergodic
and $\beta$-numbers  theories.

By~\cite{bvz}, the Julia set $\J_c$ is contained in the horizontal strip $|\Im(z)|\leq 2\sqrt{\eps}$, $\eps=|c+2|$, and 
hence $\beta$s in the large scale are uniformly  comparable  to $\sqrt{\eps}$.

Scales of Julia sets $\J_c$ bigger \cred{than $\sqrt\epsilon$} are directly affected by a large scale geometry which becomes flat when $c$ tends to $-2$; scales smaller than $\sqrt\eps$  are wiggly  with a  frequency 
that can be calculated by the Birkhoff ergodic theory. 
A typical orbit that approaches $c$ at the distance $\eps$ 
\cred{suffers 
$|\log \eps|$ corresponding scales where the wiggliness (that is, the beta number) is bigger than $\sqrt\eps$. 
Hence, the density  
of the  scales with wiggliness  bigger than $\sqrt{\eps}$  at the initial point of the orbit is at least    
$\sigma_c(\eps)|\log \eps|$.}
Taking into account  $(\ref{equ:HD1})$, the above argument suggests that the upper bound for $\HD(\J_c)$ may not be better than 
$1+C\sigma_c(\eps) |\log \eps| $. 

\begin{theo}\label{theo:hausTop}
There exists $\kappa^*> 0$ and $c_0>-2$ such that for all parameters $c\in \A \cap (-2, c_0]$, 
$$\HD(\J_c)\leq ~ 1+ \kappa^*\; |\log(c+2)|\; (c+2)^{\frac 12} \,.$$
\end{theo}

 

%

\paragraph{Lower estimate of $\HD(\J_c)$.}
The lower estimate for $\HD(\J_c)$ is based on completely different techniques from the upper estimate.

{{Cantor repellers}}, Definition~\ref{defi:rep}, form a well-known class of mappings that was studied intensively in the context of thermodynamical  formalism~\cite{si,ruelle, bow} and the relations between harmonic
and Hausdorff measures~\cite{car, mak, man, przyhar} both from ergodic and probabilistic perspectives. 

A more general class of functions than Cantor repellers is constituted by {\em box mappings}~\cite{book}, see Section~\ref{sec:rep}. An important property of
box mappings is that they have a non-trivial inner dynamics through {\em inducing} constructions proposed by M.\ Jakobson~\cite{jak}  for real unimodal maps and by J.-C.\ Yoccoz~\cite{YocJak} for unicritical polynomials. In the current paper, we will apply a few simple  inducing steps to canonically-defined  box mappings to construct a family $\vp_c$ of Cantor repellers  for every $c$ close enough to $-2$.

A novelty of our thermodynamical approach  lies in proving uniform estimates of the Hausdorff dimension of Julia sets
$\J_{\vp_c}$, initially based on the dimension of the Julia set of $\vp_c$ restricted to the real line, and then using a further branch of $\vp_c$ to improve the lower bound.  
 The thermodynamical foundations of this approach come from~\cite{ruellegaz}, we refer the reader for a more detailed discussion to~\cite{hans}.

\begin{theo}\label{theo:hausBot}
There exists $\kappa_*>0$ and $c_0>-2$ such that for every $c \in (-2, c_0]$, 
$$  \HD(\J_c)\geq 1+\kappa_*(c+2)^{\frac 12}\,.$$
\end{theo}

\paragraph{Notation and uniform constants.}
We will use the notation $A~\lesssim ~B$ to indicate that
$A/B$ is bounded from above by a uniform constant (with respect to $c \in \A$, where $\A$ is fixed following Proposition \ref{prop:realintro}).
Similarly, we define  $A~\grtsim ~B$.
Finally, \label{not:leq} $A\sim B \Leftrightarrow A~\lesssim~ B$ and
$A~\grtsim ~B$, in which case we say that $A$ and $B$ are \emph{comparable}.

Given $a<b$,  we write $\ii{a,b} := [a,b] \cap \Z$.

\subsection{Hyperbolic estimates along ray $(-\infty,-2]$.}
\label{sec:jiang}

Asymptotic estimates at $-2$ for $c<-2$ are easier than these for $c>-2$ as they correspond to hyperbolic dynamics.
They were originally proven in~\cite{jiang}, we revisit the estimates giving different proofs based on both real and complex methods.
\paragraph{Lower bound.}
An ad hoc estimate for $\HD(J_c)$, $c<-2$, comes from  Manning's formula (see \cite{car}, \cite{przyhar}) for  the
Hausdorff dimension of the harmonic measure $\omega_c$ with a base point at $\infty$, 
\begin{equation}\label{equ:harm}
\HD(\J_c)\geq \HD(\omega_c)=\frac{\log 2}{\log 2 + G_c(c)}\,,
\end{equation}
where $G_c$ is the Green function for 
$\C\setminus \J_c$; $G_c$ coincides at $c$ with the Green function for 
$\hat{\C}\setminus {\M}$.

Since ${\M}$ contains the interval
$[-2,0]$ which has logarithmic capacity $1/2$ \cite[Corollary 9.9]{pom}, $G_c(c)\leq 2\sqrt{|c+2|}$ and
$$\HD(\J_c)\geq \frac{\log 2}{\log 2 +2 \sqrt{|c+2|}}\geq 1-\frac{2\sqrt{|c+2|}}{\log 2}\,.$$

\paragraph{Upper bound \emph{via} real methods.}
The estimate that $G_c(c)\leq 2\sqrt{|c+2|}$, for $c<-2$,  can be almost inversed due to Tan 
Lei's result~\cite{tan} about the conformal similarity between $\M$ and $J_{-2}$ which implies that  asymptotically for any positive $\varepsilon$, $G_c(c)\geq |c+2|^{\frac{1}{2}+\varepsilon}$.
It was proved by purely real methods in~\cite{jiang}  that the ad hoc estimate is indeed precise. For the convenience of the reader we will present a short proof of  the upper estimate of~\cite{jiang}, 
$1 - C'\sqrt{|c+2|} \geq \HD(\J_c)$.

Let 
\begin{equation}
\label{equ:deps}
\eps=|c+2|
\end{equation}
and suppose that $c \in (-3,-2)$.  Let $p_c$ be the repelling fixed point of $f_c$ at which ${f_c}$ preserves  the orientation on the real line. 
Let $\pm y$ denote the points $f^{-1}(-p_c$), so $y \sim \sqrt{\eps}$. Denote by $I$ the interval $[-p_c,p_c]$. Since $c<-2$, $f_c^n(c)$ tends  to {\cred $+\infty$} and the Julia set is 
$$
\bigcap_{n\geq0} f^{-n}(I).$$
If $A$ is a connected component of $f^{-n}(I)$, then there are at most two iterates $k$ ($0\leq k \leq n-1$) for which $f^k(\partial A) \cap \{\pm y\} \ne \emptyset$, one of which equals $n-1$. Denote the other by $k_0$. 
Hence $f^n$ on $A$ can be decomposed as 
$$f^n = f \circ g_1 \circ f \circ g_2,$$
(or $f^n= f\circ g_1$ if $k_0$ does not exist) where $g_1$ and $g_2 = f^{k_0}$ have uniformly bounded distortion, using the Koebe distortion theorem for $g_2$ (a neighbourhood of $A$ is mapped diffeomorphically by $g_2$ onto $I$ and $g_2(A)$ is far from $\partial I$) 
and uniform boundedness of $\sum_{j=k_0+1}^{n-1} |f_c^j(A)|$ together with a standard argument using 
$$
\sum_{j=k_0+1}^{n-1} \int_{f_c^j(x)}^{f_c^j(y)} \log|Df_c(s)|\, ds
$$ for $g_1$, noting  
 $-p_c \in f\circ g_2(A)$. 


    Let $E_j = \{x \in f^j(A) : f^{n-j}(x) \in [-y,y]\}$ {\cred and remark that $f^j(E_0)=E_{j}$ for $j \leq n$}.
        The measure of $E_n = [-y,y]$ and $E_{n-1}$ are both $\sim \sqrt{\eps}$. 
        By bounded distortion (of $g_1$), $|E_{k_0+1}|/ |f^{k_0+1}(A)| \sim \sqrt{\eps}$. 
        Moreover, $E_{k_0+1}$ is  well inside  $f^{k_0+1}(A)$ in the sense that each component of $f^{k_0+1}(A)\setminus E_{k_0+1}$ has length uniformly comparable to
        $|f^{k_0+1}(A)|$.
        Pulling back once gives $|E_{k_0}|/|f^{k_0}(A)| \sim \sqrt{\eps}$  and then bounded distortion of $g_2$ gives $|E_0|/|A| \sim \sqrt{\eps}$. 
        
        Consequently,  the length of $f^{-n}(I)$ is smaller than $(1-C\sqrt{\epsilon})^n$, where $C$ is a uniform constant. Since $f^{-n}(I)$ has $2^n$ components $A$, {\cred  using H\" older's inequality with $p=\frac{1}{1-\al}$ and $q = \frac{1}{\al}$},
        $$\sum |A|^{\alpha}
        = 
        \sum 1\cdot |A|^\alpha
        \leq 2^{(1-\alpha)n}
        (1-C\sqrt{\epsilon})^{\alpha n} ~\lesssim~ 1$$
        provided $\alpha\geq 1-C'\sqrt{\epsilon}$ for some small uniform constant $C'>0$. Therefore $\HD(\J_c)\leq \alpha$. 
 
\subsection{Structure of the paper}
In the following section we recall general properties of Collet-Eckmann maps and then give properties of  Yoccoz' strongly regular parameters, those comprising the set $\A$.

Section \ref{sec:stat} contains estimates on the conformal and absolutely continuous invariant measures. 
We decompose forward orbits depending on the passages near the critical point and use ergodicity to obtain estimates on the density of blocks of different types.
 
In Section \ref{sec:dens}, density of blocks of the forward orbit are translated into densities of scales of preimages. Section \ref{sec:upp} is dedicated to the proof of the upper bound of the Hausdorff dimension using the $\be$ numbers theory. 

The last section presents the proof of the lower bound of the Hausdorff dimension by constructing a Cantor repeller inside the Julia set.

\section{Collet-Eckmann parameters}
\label{sectCE}

\subsection{Technical sequences}\label{sec:techseq}
We shall consider maps $f=f_c$ with Collet-Eckmann parameter $c \in \hat \A$, see  formula (\ref{equ:BC}), with associated constants $\Omega, \Omega'$.
We follow closely the formalism introduced in~\cite[Lemma~2.2]{nonhyp} in order to make 
precise references to estimates of~\cite{nonhyp}.
As in~\cite{nonhyp}, we introduce three positive sequences 
$(\alpha_n ),\,(\gamma_n),\,(\delta_n)$ as follows. 
\begin{defi} \label{def:sequences}
    For $n \geq 1$, set
    \begin{itemize}
        \item
$\delta_n=\frac{1}{8n^2}$,
\item
    $  \gamma_n= 64\frac{e^{ n \Omega /4}}{1-e^{-\Omega/4}},$
\item
    $
\alpha_n= (1-e^{-\Omega{\cred /4}})\,\sqrt{\delta_n\Omega'}\,  e^{ n \Omega /4}/{\cred 64}.$
\end{itemize}
\end{defi}
The growth of the derivative of  $f^{n}$ will  be given 
in terms of $\gamma_{n}$, the corresponding distortion will be bounded
by $\delta_{n}$, and various constants will be controlled through $\alpha_{n}$. 
\begin{lem}\label{lem:techseq} 
The sequences $(\alpha_n),\,(\gamma_n)$,\, and $(\delta_n)$ satisfy
$$ \begin{array}{rlrl}
\lim_{n\to\infty}\alpha_n&=~\infty~,\\
\sum_{n}\gamma_n^{-1} &<~1/64~,\\
\sum_{n}\delta_n~~&<~1/2~&\\
\end{array}$$
and, for every 
 $c \in \hat \A$,  
$$|(f_c^n)'(c)|~\ge~\alpha_n^2\,\gamma_n^{2}\,/\,\delta_n~.$$
\end{lem}

{\cred 
\begin{proof}
The bound is a direct consequence of Definitions \eqref{equ:BC} and \ref{def:sequences}.
\end{proof}
}

\subsection{Constants and scales}\label{sec:const}

A scale around the critical point $0$ of $f_c(z)=z^2+c$
is given in terms of a fixed  number $R'\ll1$ as in~\cite{nonhyp} where the general case of rational maps was studied.
We will refer to objects which stay away from the critical
points at distance $R'$ and are comparable in size
to $R'$ as the objects of the {\em large scale}. 
The proper choice of $R'$ is crucial in obtaining uniform
estimates based on the Poincar\'e series technique
of~\cite{nonhyp}.

The following  conditions define $R'$, compare the conditions $\rm{(i-iv)}$ from  \cite[Section~2.3]{nonhyp}.
\begin{description} 
\itemm{\rm{(i)}} 
fix $\tau$ so that $\alpha_k>10^3$ for all $k\ge\tau$; let $R>0$ be so small that the first return time of $0$ to $\{|z|<\sqrt{R}\}$ {\cred  by $f_c$} is at least $\tau$;
\itemm{\rm{(ii)}}  $R'$ satisfies the condition
$0 < R'~\le ~ R \inf_{n}\br{\alpha_n}^2/10^3.$ 
\end{description} 
{\cblue
        Note that $R' \ll R$ as $\alpha_1 < 1$.
        }
The constants {\cred $R$ and $R'$} above can be fixed uniformly for $c\in \hat \A$. {\cred Indeed, if
$|f_c^n(c)| < \sqrt{R}$ for some $n < \tau$ and $c\in \hat \A$, then by \eqref{equ:BC}
\[
\Om' e^{\tau \Om} \leq |(f_c^\tau)'(c)| < 4^\tau\sqrt{R},
\]
so $R$ can be chosen explicitly 
\[
R:=(\Om')^2 4^{-2\tau} e^{\tau\Om}.
\]
}

\subsection{Uniform contraction}
Replacing $R'$ by a smaller constant if necessary, the following backward contraction results hold. 
\begin{prop}\label{prop:backcon}
    There exist $\xi \in (0,1)$ and $C >0$ such that, for all $c \in \hat A$, for every $n \in \N$, 
for every ball $B(z,r)$, $z\in \J_c$ with $r \leq R'$, the diameter
of every component of $f_c^{-n}(B(z,r))$ is smaller than $\xi^n$.
Moreover, 
if $W_n$ is a component of   $f_c^{-n}(B(z, r))$, $r \leq R'$, then
\begin{equation}\label{equ:sqrt} 
\diam W_n < C \xi^n \sqrt{r}.
\end{equation}
For all $n \geq 0$, $z \in \J_c$ and $V_n$ a connected component of $f_c^{-n}(B(z,2R'))$
\begin{equation}\label{equ:small} 
\diam V_n < \frac{1}{1000}.
\end{equation}
\end{prop}
The constant $\xi$ is called the \emph{backward contraction factor}, see~\cite{przro}. For a given $c$, these estimates are known to hold; we must check that the estimates are uniform for $c \in \hat A$. The proof occupies the remainder of this subsection. 

In~\cite{nonhyp}, a decomposition of orbits into three types was introduced. Types 1 and 3 correspond to pieces
of the orbit shadowed closely by the corresponding pieces of the critical orbit.  For the reader familiar with
the work of~\cite{Becaa}, types 1 and 3 correspond to the 	bound period, the formal definition can be found in ~\cite{nonhyp}. 
Our focus will be on type $2$  preimages which correspond to the free period in~\cite{Becaa}.

\subparagraph{Second type.}
A piece of a  backward orbit is of the second type
if there exists a neighbourhood of size $R'$
which can be pulled back univalently  along the backward orbit.
Type two preimages yield expansion along pieces of orbits
of a uniformly bounded length $L$. In this setting, type $2$ corresponds
to pieces of backward orbits which   stay
at a definite distance  from the critical points. {\cred  The following lemma shows that 
$L$ can be chosen uniformly in $c \in \hat\A$ so that any univalent pullback of length at least $L$ of
a large scale ball yields a definite backward contraction.

\begin{lem}
\label{lem:type2}
Given $\la > 1$, there exist $L>1$ such that the following holds for all $c \in \hat\A$, $n \geq L$ 
and $z \in \C$ with $\dist{z,\J_c}\,\leq \,~R'/2$. If the ball $B(z, R')$ can be pulled back univalently along
a sequence $f_c^{-n}(z),\cdots, f_c^{-1}(z),z$ of preimages of $z$, then
\[
|(f_c^n)'(f_c^{-n}(z))| > \la.
\]
\end{lem}
\begin{proof}
Observe that by its definition \eqref{equ:BC}, $\hat \A$ is a compact set. Also, by the Collet-Eckmann condition 
    \eqref{equ:BC} for parameters $c \in \hat\A$, the only Fatou component of $f_c$ is its basin of attraction of infinity $A_c(\infty)$. If a point $z \in A_c(\infty)$, then there exists $\varepsilon>0$ and, for all $z'\in B(z,\varepsilon)$ and  $c' \in B(c,\varepsilon)$, $z' \in A_{c'}(\infty)$. 

Let $(c_k) \subset \hat\A$, $(z_k) \subset \C$ and a strictly increasing sequence $(N_k) \subset \N$ such that for all $k \geq 0$, $z_k \in J_{c_k}$ and the ball $B(z_k, R'/2)$ can be pulled back univalently by $f_{c_k}$ along a backward orbit of $z_k$ of length $N_k$. 

By compactness, we may assume that there exists $c \in \hat\A$ and $z \in \ol{B(0,2)}$ such that
\[
c = \lim_{k \rightarrow \infty} c_k \text{ and } z  = \lim_{k \rightarrow \infty}f_{c_k}^{-N_k}(z_k).
\]
    As $c \in \hat \A, z_k \in J_{c_k}$, $z \notin A_c(\infty)$ so $z \in J_c$. 

    For any $\varepsilon>0$, there exists an $N$ such that $f^N_c(B(z,\varepsilon))$ contains neighbourhoods of both fixed points of $f_c$, so for large $k$, $f_{c_k}^n(B(z,\varepsilon))$ contains both fixed points of $f_{c_k}$ for all $n\geq N$. Therefore $f_{c_k}^n(B(z,\varepsilon))$ is not contained in a ball of radius $R'$ for $n\geq N$. 
    Consequently, $B(z,\varepsilon)$ is not contained in a pullback of $B(z_k,R'/2)$ by $f_{c_k}^{N_k}$ for large $k$. Upon choosing $\varepsilon$ appropriately, the result follows by the Koebe 1/4 theorem. 
\end{proof}
We fix $L$ provided by the previous lemma for $\la=100$.
}

\paragraph{Contraction of preimages.}
The Collet-Eckmann condition implies the existence of  the backward contraction factor $\xi$, via  \cite[Propositions~2.1 and~7.2]{nonhyp}, which depends  only on the constant $\Omega$ from
(\ref{equ:BC}), see also the condition {\em ExpShrink} and the main result in \cite{przRLSmi03}. 
Indeed in the proof of  \cite[Proposition~7.1]{nonhyp}, we have that $\xi^{-1}$
can be taken as $\inf_n (\tilde\omega_n)^{1/n}$, where
$$\tilde\omega_n~:=~\inf
\brs{\gamma_{k_1}\,\dots\,\gamma_{k_l}\,\omega_m\,/\,16~:
~~k_1+\dots k_l+m=n}~$$
and $\omega_n$ was defined in the proof of Proposition~2.1 of~\cite{nonhyp},
$$\omega_n~:=~\inf
\brs{\,K\,\lambda^{k_0}\,\prod_{j\ge1}\gamma_{k_j}'\,:
~k_0+k_1+k_2+\dots\in{[n-L,n)}}~,$$
where $K>0$ is the expansion yielded by type 2 preimages of the length $l\in[0,L)$, $\lambda$ is the average expansion of a type 2 block and
$$\gamma_n'\,:=\,\inf\brs{\prod_{j}\gamma_{i_j}:\,i_0+i_1+i_2+\dots= n}.$$
This shows that $\xi<1$ and depends solely on $\Omega$ and $\Omega'$
in (\ref{equ:BC}).

\cred{
Remark~7.1 of~\cite{nonhyp} implies the stronger uniform estimate \eqref{equ:sqrt}.
By eventually shrinking $R'$, we directly obtain~\eqref{equ:small}. This completes the proof of Proposition~\ref{prop:backcon}.
}

\subsection{Visits to the large scale for strongly regular parameters}\label{sec:yoc}
For $c \in [-2,c_0]$, 
$c_0+2 >0$ and small, 
we denote by 
 $q_c$ the orientation-reversing fixed point of $f_c$ and $p_c$ the orientation-preserving one. 
By direct computation
\begin{equation}
\label{equ:pc}
p_c=2-\frac{\eps}{3} +O(\eps^2) \text{ and } c + p_c = \frac 23 \eps + O(\eps^2).
\end{equation}
The other fixed point, $q_c$, is in a small neighbourhood of $-1$. We have $-2<  -p_c < c < q_c < 0.$
 
{\cblue
J.-C. Yoccoz in \cite{YocJak} defines a set of \emph{strongly regular} parameters, which we shall denote by $\A_{\mathrm{Yoc}}$, and shows that $-2$ is a one-sided density point for $\A_{\mathrm{Yoc}}$. Let $c \in \A_{\mathrm{Yoc}}$.  
}
Just prior to Definiton~3.7 in \cite{YocJak}, a sequence of times $(N_k)_k$ called \emph{regular returns} is defined.
By the definition of regular returns, $f_c^{N_k-1}(c) \in (q_c,-q_c)$
 and a neighbourhood of $c$ is mapped by $f_c^{N_k -1}$ diffeomorphically onto the open interval $\widehat A$ of \cite{YocJak} ({\cred  defined such that $f_c(\pa \widehat A)=\{-q_c\}$}), which contains a $\frac12$-neighbourhood of {\cred  $A=[q_c,-q_c]$ if $c$ is close to $-2$}. 
The number $M = N_1$ depends on $c$, increasing as $c \to -2$,  $\log\eps=\log |c+2| \sim -2M\log 2$ in \cite[Proposition 3.1(3)]{YocJak}.

{\cred 
If $\eps=c+2$ is close to $0$ then the critical orbit (starting with $f_c(c)$) is ``trapped'' for some time, comparable to $M$, in a vicinity
of  $p_c$. 

\label{def:gc}
Let $g_c$ be the inverse of the restriction of $f_c$ to $\{\Re (z) > 0\}$.
Observe that $g_c$ is univalent on $B(p_c,3)$ and $g_c(B(p_c,3)) \se B(p_c,3)$. Therefore, $g_c^{k}$ is also univalent on $B(p_c,3)$ for all $k > 0$. As $R' \ll 1$, the distortion of $g_c^{k}$ is bounded on $B(p_c, 20R')$ by a constant very close to $1$ and by $2$ on $B(p_c,1/10)$.

A return $N_k$, $k \geq 1$ is \emph{simple} if for all $N_{k} < n < N_{k+1}$, 
\[
f_c^n(0) \notin A.
\]
Another characterisation of simple returns is  \cite[Lemma~3.6]{YocJak}, 
\[
N_{k+1}-N_{k} < M  - 1.
\]
By Remark 3.9 in \cite{YocJak}, all returns with $N_k < 2^{\sqrt M}M$ are simple. 

Let us denote $c_k := f_c^{N_k}(0)$. Assume that $N_k$ is a simple return and let $m=N_{k+1}-N_{k}$. Then by the distortion bounds above (and few iterates outside $B(p_c,1/10)$), we have that for all $j \in \ii {0,m-1}$
\[
|(f_c^{j})'(f_c(c_k))| \sim |2p_c|^j \text{ and } f_c(c_k)+p_c \sim |2p_c|^{-m}.
\]
By estimate \eqref{equ:pc}, for all $j \in \ii{0,\eps^{-1}}$,
\[
|2p_c|^j = 4^j\left(1-\frac \eps 6 + O(\eps^2) \right)^j \sim 4^j.
\]
As $m < M - 1 \sim |\log \eps| \ll \eps^{-1}$, combining the previous two estimates, for all $j \in \ii{0,m-1}$,
\begin{equation}
\label{equ:4}
|(f_c^{j})'(f_c(c_k))| \sim 4^j \text{ and } f_c(c_k)+p_c \sim 4^{-m}.
\end{equation}
As $N_1 = M$ and $m \leq M - 2$, by estimate \eqref{equ:pc}
\[
-p_c < c < -f_c(c) < f_c(c_k) \text{ and } |f_c(c)+c| > \frac 52 |c+p_c| > \eps.
\]
Pulling back once by $f_c$, for all $j \in \ii{1,m}$ we have
\begin{equation}
\label{equ:2}
\sqrt{\eps} < |c_k| \sim 2^{-m} \text{ and } |(f_c^j)'(c_k)| \sim 2^{2j-m}.
\end{equation}
}

{\cblue
\begin{lem} \label{lem:simplestart}
    Given $\delta >0$, there 
     exists  $c_0 > -2$ such that, if $c \in 
     [-2,c_0]$ and 
     the first $k$ return times  $N_1, \ldots, N_k$ of $c$ to $A$  are simple, 
     then
     $$
    \left| (f_c^{j})'(c) \right| \geq e^{j(\log2 - \delta)} 
    $$
    for all $j = 1, \ldots, N_k$. 
\end{lem}
}
\begin{proof}
    \cblue{
Let $C_*>4$ be a universal multiplicative constant  which  makes \eqref{equ:2} hold. 
Choose $N$ large enough that $e^{N\delta} > C_*^2$. 
    Let $\eps$ be small enough that $e^{M\delta} > C_*^{2N}.$
        Then $$4^M = 2^M e^{M\log2} > 2^M e^{M(\log 2 - \delta)} C_*^{2N}.$$

        By the conjugacy between $f_{-2}$ and the full tent map,  if 
        $x\in [-1,1]$ and $ f_{-2}^n(x) \in B(0, 1.1)$ then (see \cite[Section~2.2]{YocJak}), for all $n\geq 1,$
        $$
        \left| (f_{-2}^n)'(x) \right| \geq {2^n}
        \frac{1}{2}.
        $$
        By continuity, if $c$ is close enough to $-2$, if $x, f_c^n(x) \in A$, and $n \leq N$
        then 
        \begin{equation}\label{equ:shortN}
        \left| (f_{c}^n)'(x) \right| \geq {2^n}
        \frac{1}{4}.
        \end{equation}
    }
        We have $N_1= M$ and 
     $$
    \left| (f_c^{j})'(c) \right| \geq \max(2^j, 4^j/C_*) 
    $$
    for $j \leq M-1$. 

    Now we examine times $N_{k_j}$ where $k_0=1$ and $k_{j+1}$ is the minimal $k'>k_j$ with $N_{k'}-N_{k_j} >N$. 
    Let $m_j = N_{j+1}-N_j \geq 2$ and $m = N_{k_{j+1}}-N_{k_j}$. 
    We have that  $m - m_{k_{j+1}-1} \leq N$.  
    Thus we can split up the iterates from $c_{k_j}$ until $c_{k_{j+1}}$ into one sequence of simple returns with combined length at most $N$, and one single simple return. Applying \eqref{equ:shortN} and \eqref{equ:2}, we obtain
     $$
    \left| (f_c^{m})'(c_{k_j}) \right| \geq 2^{m} \frac{1}{4C_*}.
    $$
    As $e^{N\delta}> C_*^2$ and $m > N$, 
     $$
    \left| (f_c^{m})'(c_{k_j}) \right| \geq e^{m (\log 2 - \delta)} .
    $$
    This proves the lemma at all times $N_{k_j}$.

    For intermediate times, we use $k_{j+1}-k_j < N$ 
    and the initial growth estimate 
     $$
    \left| (f_c^{M-1})'(c) \right| \geq  4^M/C^2_* > 2^M e^{M(\log 2 - \delta)}
    $$
        and apply \eqref{equ:2} repeatedly. 
\end{proof}

\begin{proof}[Proof of Proposition~\ref{prop:realintro}]
With the notation introduced above, let $m = N_{k+1}-N_k$ for some $k  \geq 1$.
    From \cite[equation~(3.2)]{YocJak} {\cred if $m > M$}, 
    \begin{equation} \label{equ:mM}
        m \leq 2^{-\sqrt{M}}(N_{k}+m).
    \end{equation}
    If  we set $\tau(\eps) =  \exp(-\kappa\sqrt{\log1/\eps})$, for $\kappa$ small enough, then $m \leq N_k\tau(\eps)$,
    which yields {\cred the bound on $n^*$ in} Proposition~\ref{prop:realintro} {\cred  if $c_0+2$ is small enough} and $c \in \A_{\mathrm{Yoc}}$.
        
    {\cblue
    Now we prove the estimate on
     the growth of the derivative  along the critical orbit of $c \in \A_{\mathrm{Yoc}}$.  
    }
    By \cite[Proposition~3.10]{YocJak}
    {\cblue 
    with $g_{B(k)}$ satisfying $f_c^{N_k-1} \circ g_{B(k)}=\mathrm{Id}_{\widehat A}$ and $x=c_k$,
    }
    $$\left|\log \left(|(f_c^{N_k -1})'(c)|\frac{h_c(c)}{h_c(c_k))} \right) - (N_k-1)\log2\right| < CM^{-1}N_k,$$
        where $\log h_c(c) \sim M\log 2$ and $h_c(c_k) \sim 1$; $C$ is a universal constant. 
        In particular, at each such time, 
        $$
        \left|\frac{1}{N_k-1}\log |(f_c^{N_k-1})'(c)| - \log2\right|  < \frac{C}M\frac{N_k}{N_k-1} + \frac{C'\, M}{N_k-1}\,,$$
where $C'>0$ is an universal constant.
The returns being regular, the Koebe distortion theorem gives a lower bound $C_1>0$ to the derivative of $f_c^m$ at $c_k$. 
        Hence for any $1\leq j \leq m$, $|(f_c^j)'(c_k)| \geq 4^{-m}C_1$. 
    We also deduce that
    \begin{multline*}
        \left|\frac{\log |(f_c^{N_k+j-1})'(c)|}{N_k+j-1} - \log2\right|  < \frac{C}M\frac{N_k}{N_k-1}  + \frac{C'\, M}{N_k-1} + \\
        \left| \frac{1}{N_k+j-1}\log\frac{|(f_c^{j})'(c_k)|}{2^j} \right|.
        \end{multline*}
        {\cblue
        The last term above is bounded by $4m/N_k$ and allows us to treat the case
           $N_k > 2^{\sqrt M -2}M$. In this case,  
        $$4 m/N_k \leq  2^{4-\sqrt{M}}(1-2^{2-\sqrt M}) \ll M^{-1},$$
        which is immediate if $m\leq M$, while for $m>M$ we use  \eqref{equ:mM}.
        As the first non-simple return happens at a time at least $2^{\sqrt M }M$ and simple returns have return time bounded by $M$, there is a regular return $N_k$ with 
        $$2^{\sqrt M -2}M < N_k <  
        2^{\sqrt M -1}M. $$
        Consequently, for all  $n \geq 
        2^{\sqrt M -1}M$,
\begin{equation}\label{equ:lypyoc}        
        \left|\frac1n \log |(f_c^n)'(c)|-\log 2\right| < \frac{C_2}{M}~ \sim~ \frac{1}{\log \frac{1}{\eps}}, 
        \end{equation}
 where $C_2$ is a universal constant. Given $\delta >0$, for large $M$, $C_2/M < \delta$. 
    }
    To finish the proof, note that 
     regular return times up to some $N_k > 2^{\sqrt{M}-1 }M$ are simple and apply Lemma~\ref{lem:simplestart}.
\end{proof}

{\cred
Let \[
k_0 := \max\{k > 0\ :\ N_k < 2^{\sqrt M} M\}.
\]
For each $k \in \ii{1,k_0}$, let $\epsilon_k$ be maximal such that 
\[
B(c,\epsilon_k) \se f_c^{1-N_k}(B(c_k,R')).
\]
By bounded distortion and Schwarz lemma 
\begin{equation}
\label{equ:epsk}
\epsilon_k |(f_c^{N_k-1})'(c)| \in (R'/2,R')
\end{equation}
and $\eps_1 \sim \eps R' \sim \eps$. 

By estimates \eqref{equ:2}, for all $k \in \ii{2,k_0}$
\begin{equation}
\label{equ:scales}
 \sqrt \eps ~\lesssim~ 2^{N_{k-1}-N_k} \sim \frac{\eps_k}{\eps_{k-1}} \sim |c_{k-1}|.
\end{equation}
Let us denote
\[
r_c := \eps_{k_0} .
\]

\begin{lem}\label{lem:rc}
For any $t>1$, if $\eps = c+2$ is sufficiently small and $c \in \A$, then
\[
r_c < \eps^t.
\]
\end{lem}
\begin{proof}
By choice of $\A$ following Proposition~\ref{prop:realintro}, taking $\de = 1/10$ and $c \in \A$, 
for
    all $n \geq 0$,
\[
|(f_c^n)'(c)| \geq e^{n(\log 2 - \de)}.
\]
Combine this with bound \requ{epsk} to get 
\[
r_c < R'e^{(1-N_{k_0})(\log2 - \de)},
\]
while bound \requ{4} for $k=0$ gives $\eps \geq e^{-2M \log 2 - C}$, for some uniform constant $C>0$.

By inequality \requ{mM}, 
\[
N_{k_0} \geq M(2^{\sqrt M}-2) > 1 + t\frac{2M \log 2 + C}{\log2 - \de},
\]
if $M$ is large. The above bounds on $\eps$ and $r_c$ are sufficient to conclude.
\end{proof}
}

\section{Statistical methods}
\label{sec:stat}
\subsection{Conformal measure of disks centred at $c$}
Conformal or {\em Sullivan-Patterson} measures are  dynamical analogues
of  Hausdorff measures in dynamical systems.
\begin{defi}\label{def:mes}
Let $f_c=z^2+c$ be a  rational map with the Julia set $\J_c$.
A Borel measure $\nu$ supported on $J$ 
is called {\em conformal with an exponent} $p$
(or {\em $p$-conformal}) if for every Borel set
$U$ on which $f_c$ is injective one has
\[ \nu(f_c(U))=\int_{U}|f_c'(z)|^{p}~d\nu~.\] 
\end{defi}
As observed in \cite{sul}, 
the set of pairs $(p,\nu)$ with $p$-conformal measure $\nu$
is compact (in the weak-$*$ topology). 
Hence, there  exists  a conformal measure with the {\em minimal exponent} 
\[
\dconf(c)~:=~\inf\{p: \exists\; \mbox{a $p$-conformal measure on $\J_c$}\}.
\]  
The minimal exponent $\dconf(c)$ is also called 
the {\em conformal dimension} of $\J_c$.

The following fact, proven in~\cite{nonhyp} {\cred (Theorems 3 and 7)},
explains basic properties of conformal measures
for Collet-Eckmann parameters.
Several claims of the following two facts were known (in the Collet-Eckmann setting) earlier, see \cite{prz98}.

\begin{fact}\label{theo:4}
Let $c \in \A$. Then there is a unique,
ergodic, non-atomic, and  probabilistic conformal measure $\nu_c$ for $f_c$,
with exponent
$$\dconf(c)=\HD{\J_c}.$$
Moreover,
$$\HD(\nu_c)=\HD{\J_c}\,.$$
\end {fact}


 \paragraph{Preliminaries.}
{\cred
By compactness, the eventually onto property and Definition \ref{def:mes}, we have
\[
\nu_0:=\inf\{\nu_c(B(x,R'/2))\ :\ c \in \A \text{ and } x \in [-p_c, p_c]\} > 0.
\]
Therefore, whenever a branch of $f_c^{-n}$ is univalent on some $B(x,R')$, $x \in \J_c$, there is a constant $C_\nu>1$ depending only on $\nu_0$ and $R'$, such that if $W$ is the corresponding connected component of $f_c^{-n}(B(x,R'/2))$
\begin{equation}
\label{equ:nuW}
C_\nu^{-1} (\diam W)^{\dconf(c)} \leq \nu_c(W) \leq C_\nu (\diam W)^{\dconf(c)}.
\end{equation}
Moreover, by the bounded distortion of $f_c^n$ on $W$, there is a universal constant $K_2>1$ such that if $W \ni y:=f_c^{-n}(x)$, $r:=K_2^{-1}\diam W$ and $r':=K_2\ \diam W$,
\[
\nu_c(B(y,r)) \leq C_\nu K_2^{\dconf(c)}r^{\dconf(c)}
\]
and
\[
C_\nu^{-1} K_2^{-\dconf(c)}r'^{\dconf(c)} \leq \nu_c(B(y,r')).
\]
We apply this now near $p_c$.
    As $p_c$ is a repelling fixed point, 
    any ball of radius at least comparable to (in particular, at least equal to $R'$ times) the distance from its centre to $p_c$ gets mapped by an iterate of $f_c$ with bounded distortion to the large scale. Applying this and using conformality, we obtain that, for all $0<r<2$ (and $c \in \A$)
\[
\nu_c(B(p_c,r)) \sim r^{\dconf(c)},
\]
and for all $\eps R' \leq r \leq 2$
\begin{equation}
\label{equ:nularge}
\nu_c(r) \sim r^{\dconf(c)},
\end{equation}
    where $\nu_c(r):=\nu_c(B(c,r)) = \nu_c(B(-c,r))$ and we
    used the $\sim$ notation for uniform constants from page \pageref{not:leq}.

As an immediate consequence, for all $x \in [-p_c,p_c] \sm (-\sqrt{\eps R'},\sqrt{\eps R'})$,
\begin{equation}
\label{equ:cent}
\nu_c(B(x,|x|/2)) ~ \lesssim ~ |x|^{\dconf(c)}.
\end{equation}

Recall the definition of $k_0$ and $\eps_k$ before bound \requ{epsk}.
For all $k \in \ii{1,k_0}$, by the bounded distortion of $f_c^{N_k-1}$,
\begin{equation}
\label{equ:nuepsk}
\nu_c(\epsilon_k) \sim \epsilon_k^{\dconf(c)}.
\end{equation}
	\paragraph{General estimates.}
We want to obtain a sharp upper bound for $\nu_c(r)$ for all $r>0$. We have to distinguish two cases depending on the range of $r$.
}

\begin{lem}\label{lem:c}
There exists a uniform constant $S \geq 1$ {\cred such that for every
$0 < r \leq r_c=\eps_{k_0}$} and $c\in {\A}$,
$$r^{\dconf(c)(1+S\tau(\eps))}~\lesssim ~ \nu_c(r)~\lesssim~ r^{\dconf(c)(1-S\tau(\eps))},$$
where $\tau(\eps)=\exp(-\kappa\sqrt{\log 1/\eps})$
comes from Proposition~\ref{prop:realintro}.
\end{lem}
\begin{proof}

{\cred 
We show the upper bound. 
    Denote by $U_k$ the level $(N_k-1)$ univalent pullback of $B(c_k, R'/2)$ to $c$. 
    Let $k \geq k_0$ be maximal with $U_k \supset B(c,r)$. 
    With this choice, 
    $$\nu_c(r) \leq \nu_c(U_k) \sim \diam(U_k)^{\dconf(c)}.$$ 
    Choose $k'\geq k+1$ minimal with $N_k' - N_k \geq L$ 
(the constant $L$ was fixed after the proof of Lemma \ref{lem:type2}). 
    Then
    $U_{k'} \subset U_k$ 
    by Lemma \ref{lem:type2} and 
    $$\diam(U_{k'}) ~\lesssim~r ~\lesssim~ \diam(U_k).$$
    It remains to bound the ratio of $\diam(U_k)$ and $\diam(U_{k'})$. 
 By the definitions of $k_0$ before bound \requ{epsk} and of $\tau(\eps)$ in Proposition \ref{prop:realintro},
    $$N_{k'}\leq (N_k+L)(1+\tau(\eps)) .$$
    By bounded distortion, 
    $$\frac{\diam(U_{k'})}{\diam(U_k)} 
    ~\grtsim~ 
    \frac{\diam(f_c^{N_k-1}(U_{k'}))}{R'} 
    ~\grtsim~ 4^{-(N_{k'}-N_k)} ~\grtsim~ 4^{-N_k\tau(\eps)}.
    $$
    Now $4^{-N_k} = \xi^{S N_k}$, where $S = -\log4/\log \xi$ and $\xi$ is the backward contraction factor of \eqref{equ:sqrt}. Meanwhile, $\xi^{N_k} > r$ by choice of $k$. 
    Hence
    $$\frac{\diam(U_{k'})}{\diam(U_k)} 
    ~\grtsim~  r^{S\tau(\eps)}.$$
    We conclude that 
    \begin{eqnarray*}
        \nu_c(r) &\leq& \nu_c(U_k) ~\lesssim~ \diam(U_k)^{\dconf(c)} 
        \\
        &~\lesssim~ &
    \diam(U_{k'})^{\dconf(c)}r^{-\dconf(c)S\tau(\eps)} ~\lesssim~ r^{\dconf(1-S\tau(\eps))}
    .
    \end{eqnarray*}
    The lower bound is proved by a similar method.

}
\end{proof}
\paragraph{Sharp estimates for intermediate scales.}
{
\cred 
For scales larger than $r_c$, we need first to estimate $\nu_c(B(0,2\sqrt \eps))$. For any $0<r<r'$, let us denote $A(r,r'):=B(0,r')\sm B(0,r)$ the annulus around the critical point of given radii. Observe that by Definition \ref{def:mes}
\begin{equation}
\label{equ:an}
\nu_c(A(r_1,r_2)) \leq 2^{1-\dconf(c)}r_1^{-\dconf(c)}\nu_c(r_2^2).
\end{equation}
This bound will allow us to transfer estimates near $c$ to estimates for annuli around $0$. The times $N_k$ will allow us to transfer estimates near $0$ to estimates around $c$. We do so repeatedly in the following bootstrapping argument. 
\begin{lem}\label{lem:nu0}
There exists a constant $C>0$ such that for all $c \in \A$ and $r \geq \sqrt \eps$
\[
\nu_c(B(0,r)) \leq C r^{\dconf(c)}.
\]
\end{lem}
\begin{proof}
    First, using bounds \eqref{equ:nularge} and \eqref{equ:an}, if $r \geq \sqrt{\eps_1} ~\grtsim~ \sqrt{\eps R'}$
\[
\nu_c(A(r,2r)) ~\lesssim~ r^{\dconf(c)},
\]
which, as $\dconf(c) \ge 1$, sums to
\begin{equation}
\label{equ:anl}
\nu_c(A(\sqrt{\eps_1},r)) ~\lesssim~ \sum_{n \ge 1} \left(2^{-n} r \right)^{\dconf(c)} \leq r^{\dconf(c)}.
\end{equation}
By bounds \eqref{equ:an} and \requ{scales}, we have
\[
\nu_c(A(\sqrt{\eps_{k+1}}, \sqrt{\eps_k})) ~\lesssim~ \eps_{k+1}^{-\dconf(c)/2}\eps_k^{\dconf(c)} ~\lesssim~ \eps^{-\dconf(c)/4}\eps_k^{\dconf(c)/2}.
\]
Summing up, together with bound \eqref{equ:anl}, we obtain
\begin{equation}
\label{equ:anl2}
\nu_c(A(\sqrt{r_c},\sqrt \eps_1)) ~\lesssim~ \eps^{\dconf(c) / 4}.
\end{equation}
For $r \leq \frac 12 \sqrt{r_c}$, by Lemma \ref{lem:c} and bound \eqref{equ:an}
\[
\nu_c(A(r,2r)) ~\lesssim~ r^{\dconf(c)(1-S\tau(\eps))},
\]
which sums to 
\begin{equation}
\label{equ:anl3}
    \nu_c(B(0,\sqrt{r_c})) ~\lesssim~ r_c^{\dconf(c)(1-S\tau(\eps))},
\end{equation}
    a negligible quantity compared to the estimate we have for  $\nu_c(A(\sqrt{r_c},\sqrt \eps))$, as $r_c \ll \eps^2$ by Lemma \ref{lem:rc} and $\tau(\eps) \ll 1$.

    Combining \eqref{equ:anl}-\eqref{equ:anl3}, we obtain that for all $r \in [\eps^{1/4}, 2]$,
\begin{equation}
\label{equ:nu0l}
\nu_c(B(0,r)) ~\lesssim~ r^{\dconf(c)}.
\end{equation}
    In order to extend the range of validity to the desired $r \in [\sqrt\eps, 2]$, it suffices to improve the exponent $\dconf(c)/4$ in \eqref{equ:anl2} to $\dconf(c)/2$. This task occupies  the remainder of the proof. 

Let 
\[
\Lambda:=\brs{r \in (0,2]\ :\ \nu_c(r) ~\lesssim~ r^{\dconf(c)}}.
\]
    We have already proven in \requ{nuepsk} that $\eps_k \in \Lambda$ for all $k \in \ii{1,k_0}$ and estimate \eqref{equ:nularge} shows that $[\eps R',2] \se \Lambda$. We need to show that $[\eps_{k+1}, \eps_k]\in \Lambda$. 
For each $k \in \ii{1,k_0 -1}$, we distinguish three cases:
\begin{enumerate}
\item $|c_k| \geq R'$;
\item $|c_k| \in (\eps^{1/4},R')$;
\item $|c_k| \leq \eps^{1/4}$.
\end{enumerate}
    The first case is the easiest to treat, as $\eps_{k+1} \sim \eps_k$ by equation \eqref{equ:scales}, so $[\eps_{k+1}, \eps_{k}] \se \Lambda$. 
    
    To treat case 2, we pull back $B(0,2r) \supset B(c_k,r)$ by $f_c^{1-N_k}$ for all $r \in [|c_k|, R']$
    and, using estimate \eqref{equ:nu0l}, we get that 
\[
[|c_k|\eps_k,\eps_k] \se \Lambda.
\]
    By estimate \eqref{equ:scales}, we have that $|c_k|\eps_k \sim  \eps_{k+1}$ so 
    $[\eps_{k+1}, \eps_{k}] \se \Lambda$. 

    For case 3, we again pull back $B(0,2r) \supset B(c_k,r)$ by $f_c^{1-N_k}$, this time for all $r \in [\eps^{1/4}, R']$, to obtain
\[
    [\eps^{1/4}\eps_k,\eps_k] \se \Lambda.
\]
    By estimate \eqref{equ:scales}, 
    $\eps^{1/4} \eps_k ~\lesssim~ \sqrt{\eps_k\eps_{k+1}}  $
so
$
[\sqrt{\eps_k\eps_{k+1}}, \eps_k] \se \Lambda.$

For $r \in [(\eps_{k+1}\eps_{k})^{1/4}, \sqrt{\eps_k}/2]$, as $4r^2 \in \Lambda$, we bound $\nu_c(A(r,2r))$ by \eqref{equ:an} and sum up to get
\[
    \nu_c(A((\eps_{k+1}\eps_{k})^{1/4}, \sqrt{\eps_k})) ~\lesssim~ \eps_k^{\dconf(c)/2}.
\]
We use estimate \eqref{equ:an} for $A(\sqrt{\eps_{k+1}},(\eps_{k+1}\eps_{k})^{1/4})$ to obtain the same upper bound for its measure, thus
\[
\nu_c(A(\sqrt{\eps_{k+1}}, \sqrt{\eps_k})) ~\lesssim~ \eps_k^{\dconf(c)/2}.
\]
Summing for $k \in \ii{1,k_0 -1}$ provides the desired estimate. 
\end{proof}

\begin{lem}\label{lem:cl}
Uniformly in $c \in \A$, for all $r_c \leq r \leq 2$,
\[
\nu_c(r) ~\lesssim~ r^{\dconf(c)}.
\]
\end{lem}
\begin{proof}
Remark that, using the notations from the proof of Lemma \ref{lem:nu0}, it is enough to show that $\Lambda \supset [r_c,2]$. As its conclusion strengthens bound \eqref{equ:nu0l}, we can replace $\eps^{1/4}$ by $\sqrt \eps$ in the definition and proof of the three cases treated there. By the bound \requ{2}, $\sqrt \eps < |c_k|$,
so the third case becomes void. The same argument, when $\sqrt \eps$ is substituted for $\eps^{1/4}$ proves that in the other two cases, for all $k \in \ii{1,k_0 -1}$, $[\eps_{k+1},\eps_k] \se \Lambda$. As $\eps \sim \eps_1$, bound \eqref{equ:nularge} completes the proof.
\end{proof}
}

For each $c \in \A$, $\eps=c+2$ let us define
\begin{equation}\label{equ:defeta}
\eta(\eps)	:= 1-S\tau(\eps),
\end{equation}
where $S$ and $\tau(\eps)$ are given by Lemma~\ref{lem:c}.

\begin{prop}\label{prop:eta}
There exists a uniform constant  $C>1$ so that 
for every $\zeta\in (1,2\eta(\eps))$, $c \in \A$ and $r\in (r_c,R')$
\begin{equation}\label{equ:integr1}
		C^{-1}\int_{B(c,r)} 
		{|x-c|^{-\dconf(c)\zeta/2}}\,d\nu_c(x) \leq
		\frac{r^{1-\zeta/2}}{2-\zeta}+ \frac{r_c^{\eta(\eps)-\zeta/2}}{2\eta(\eps)-\zeta}\, .
\end{equation}  
\end{prop}
\begin{proof}
	We put $n_c$ to be the smallest $n$ such that $r/2^n\leq r_c$,
	hence $n_c\sim \log \frac r{r_c}$.

	Define $A_n=\{x\in \C:2^{-(n+1)}\leq |x-c|/r < 2^{-n}\}$,
	$n\in \N$, and set $I_r:=\int_{B(c,r)} |x-c|^{-\dconf(c)\zeta/2}(x)d\nu_c$. 
Splitting $B(c,r)$ into annuli $A_n$, $n\geq 0$,  and using that $\nu_c$ is non-atomic measure by Fact \ref{theo:4}, 
	\begin{eqnarray*}
		I_r
		&\leq & 4r^{-\dconf(c)\zeta/2} \left(\sum_{n=0}^{n_c}2^{n\dconf(c)\zeta/2}\nu_c(A_n)+ \sum_{n=n_c}^{+\infty} 2^{n\dconf(c)\zeta/2}\nu_c(A_n) \right)\\
		\end{eqnarray*}
 For all $n\leq n_c$, we use the sharp estimate of 	Lemma~\ref{lem:cl},
 $$\nu(A_n)\leq \nu(B(c,2^{-n}r))~\lesssim~r^{\dconf(c)} 2^{-n\dconf(c)}$$
 while for all $n\geq n_c$, the general estimate of Lemma~\ref{lem:c} gives
  $$\nu(A_n)~\lesssim~r^{\dconf(c)\eta(\eps)} 2^{-n\dconf(c)\eta(\eps)}.$$
 Recall $\dconf(c)>1$ for  $c\in \cal A$ and $\zeta/2< \eta(\eps)<1$.  Also, $r2^{-n_c}\leq r_c$. 
	\begin{eqnarray*}
I_r		&\lesssim& r^{1-\zeta/2} \sum_{n=0}^{n_c}2^{-n(1-\zeta/2)}+ r^{\eta(\eps)-\zeta/2}
\sum_{n=n_c}^{+\infty} 2^{-n({\eta}(\eps)-\zeta/2)}\\
		&\lesssim &~  \frac{r^{1-\zeta/2}}{1-2^{-(1-\zeta/2)}}
		+ (r2^{-n_c})^{\eta(\eps)-\zeta/2}
		\sum_{n=0}^{+\infty} 2^{-n({\eta}(\eps)-\zeta/2)} \;\\
		&\lesssim &~  \frac{r^{1-\zeta/2}}{2-\zeta}
		+
		\frac{r_c^{\eta(\eps)-\zeta/2}}{1-2^{-({\eta}(\eps)-\zeta/2)}}
		~\lesssim~  \frac{r^{1-\zeta/2}}{2-\zeta}+ \frac{r_c^{\eta(\eps)-\zeta/2}}{2\eta(\eps)-\zeta}\;.
	\end{eqnarray*} 
\end{proof}

\subsection{Absolutely continuous invariant measures}
Absolutely continuous invariant measures can exist only with respect to conformal measures without atoms at critical points. This necessary condition is satisfied for the geometric measures of $f_c(z)=z^2+c$, $c\in \A$.

{\cred We refer to Theorem 4 in \cite{nonhyp} for the following result.}
\begin{fact}\label{theo:inv}
Let $c \in \A$.
Then  $f_c$ has a unique absolutely continuous invariant probabilistic measure $\sigma_c$
with respect to the conformal measure $\nu_c$ from Fact~\ref{theo:4}.
Moreover, $\sigma_c$ is ergodic, exact, and has positive entropy and Lyapunov exponent.
\end{fact}

\cred{
To prove the following proposition, we will need to control distortion; we will use the method of \emph{shrinking neighbourhoods}, introduced in \cite{prz95}, see also \cite{przuz}. 
With our technical sequence $(\delta_n)$ of Definition~\ref{def:sequences}, 
let
 $\del_n\,:=\,\prod_{k\le n}\br{1-\delta_k}$.
Let $B(z, r)$ be the ball of radius $r$ around a point $z$ and $\{f^{-n}(z)\}$ 
be a sequence of preimages of $z$. We define
$U_{n}$ and $U'_{n}$ as the connected
components of $~f^{-n}(B(z, r\del_n))$ and 
$f^{-n}(B(z, r\del_{n+1}))$, respectively, which contain 
$ f^{-n}(z)$.
Clearly,
$$f(U_{n+1})~=~U_n'~\subset~U_n~.$$
If $U_k$, for $1\le k\le n$, do not
contain critical points then the distortion of
$f^n\,:~U_n'\to B(z, r\del_{n+1})$ is
bounded (by the K{o}ebe distortion theorem) 
by a power of $\frac 1{\delta_{n+1}}$, multiplied
by an absolute constant.

Since $\sum_n\delta_n\,<\,\frac 12$, one also has 
$\prod_n\br{1-\delta_n}\,>\,\fr12$,
and hence always $B(z, r/2) \subset B(z, r\del_n)$.
}

Let $\rho_c(x)=\frac{d\sigma_c}{d\nu_c}(x)$ be the Radon-Nikodym derivative of $\sigma_c$ with respect to $\nu_c$.

\begin{prop}\label{prop:integr}
There exists a uniform constant  $C>0$ so that for every $0<r\leq \diam J_c$, $\zeta\in (1,2)$  and  $c\in \A$, 
\begin{equation}\label{equ:integr}
\int_{B(c,r)} \rho_c(x)^{\zeta} d\nu_c\;\leq \;
C\int_{B(c,r)} |x-c|^{-\dconf(c)\zeta/2}\,d\nu_c(x).
\end{equation}  
\end{prop}

In particular, by  Proposition \ref{prop:eta}, the densities $\rho_c(x)$ of $\sigma_c$ with respect to $\nu_c$  are uniformly $L^{\zeta}$-integrable.
Just $L^{\zeta}$-integrability was proven before in the non-uniform setting in~\cite{nonhyp} and~\cite{rivshen} for large classes of rational functions. 

\begin{proof}
By splitting the integral into two integrals, one over the set where $\rho_c(x) \leq |x-c|^{-\dconf(c)/2}$ and one over its complement in $B(c,r)$, it suffices to show the proposition when $\zeta$ is close to $2$.

A starting point is a general upper estimate of $\rho_c(x)$
obtained in~\cite[Proposition~10.1]{nonhyp}.
Assume that $c\in \A$.
    {\cblue
    Let us set
    $$\hat \Delta_{k}(x):=\left(\dist{f_c^{k}(0), x}\right)^{-\dconf(c)/2}$$ and 
\[{\mathrm g}_c(x):= \sum_{k=1}^{\infty}
    \gamma_{k}^{-\dconf(c)}\hat \Delta_{k}(x)\, ,\]
    where $\gamma_k$ is defined in Definition~\ref{def:sequences}.
 The \cite[Proposition~10.1]{nonhyp} asserts that 
there exists a positive constant $K$ so that, 
    for all $c \in \A$ and every
$x\not \in \bigcup_{n=0}^{\infty} f_c^{n}(c)$,
\begin{equation}\label{equ:rho}
\rho_c(x) < K\; {\mathrm g}_c(x)~.
\end{equation}
The sequence $\gamma_{k}$
(defined in Definition~\ref{def:sequences}), independent of $c \in \A$, 
tends exponentially fast  to $\infty$.

    Let $\zeta\in [3/2,2)$ be an arbitrary number. We use the H\"older inequality similarly to the proof of~\cite[Corollary~10.1]{nonhyp}, for positive sequences $\mathbf{x}, \mathbf{y}$,
\begin{eqnarray*}
    \| \mathbf{x}  \mathbf{y}\|_1 & \leq &  \|\mathbf{x}^{{1-1/\zeta}}\|_\frac{1}{1 - 1/\zeta} \cdot \|\mathbf{x}^{1/\zeta}  \mathbf{y}\|_{\zeta} \\
&\lesssim & 
    \|\mathbf{x}^{1/\zeta} \mathbf{y}\|_{\zeta}  = \|\mathbf{x}\mathbf{y}^{\zeta}\|_1^{1/\zeta*},
    \end{eqnarray*}
    provided $\mathbf{x}$ is geometrically decreasing.

    With this model, we deduce from \eqref{equ:rho} that
\begin{eqnarray}
    \int_{B(c,r)} \rho_c(x)^{\zeta}d\nu_c &\leq & K^{\zeta}
\int_{B(c,r)} \left(\sum_{k=1}^{\infty} \gamma_{k}^{-\dconf(c)}\hat\Delta_{k}(x)\right)^{\zeta}d\nu_c\nonumber\\
&\lesssim &~  \sum_{k=1}^{\infty} \gamma_{k}^{-\dconf(c)} \int_{B(c,r)} \hat\Delta_{k}^{\zeta}(x)d\nu_c \,. 
\label{equ:dense0}
\end{eqnarray}

     Suppose that $r \leq R'/6$} and let $k>1$. If $f_c^k(0)\not \in B(c,2r)$ then
    we have a direct bound       
    \begin{eqnarray}
    	\int_{B(c,r)} \hat\Delta_{k}^{\zeta}(x)d\nu_c &\leq& 	\int_{B(c,r)}{\hat \Delta_{1}^{\zeta}(y)}\; d\nu_c(y)\nonumber,
    	\end{eqnarray}
    because for all $x\in B(c,r)$, $|x-c|\leq |x-f^{k}(0)|$.
    
If $f_c^k(0) \in B(c,2r)$, put $B_k:=B(f_c^k(0),3r)$ and let 
\[
I_{B_k}:= \int_{B_k} \hat \Delta_{k}^{\zeta}(x)d\nu_c \geq \int_{B(c,r)} \hat\Delta_{k}^{\zeta}(x)d\nu_c.
\]
We estimate $I_{B_k}$ using shrinking neighbourhoods. Given $x\in B_k \sm \bigcup_{n=1}^{k} f_c^{n}(0)$ let $r(x)>0$ be minimal such that some shrinking neighbourhood $U_m$, $m=m(x)\leq k$, for $B(x,r(x))$ hits the critical point ($0 \in \pa U_m$). By construction, $r(x)<2\, \dist{f_c^k(0),x} \leq 6r \leq R'$ and $r(x)$ is comparable with $\dist{f_c^{m(x)}(0), x}$.
We can write $B_k \sm \bigcup_{n=1}^{k} f_c^{n}(0) = \bigcup_ {m=1}^k E_m$, where $E_m=\{x\in B_k:m(x)=m\}$. 
For some $m$, $E_m$ may be empty.

By Lemma~2.3 of~\cite{nonhyp}, $\dist{c, f^{-m+1}(x)}\leq 6r \gamma_{m-1}^{-1} < r$, as $\gamma_{m-1} > 64$ by Definition~\ref{def:sequences}. Hence, for all $m\leq k$, 
$${ f_c^{-m+1}(E_m)}\subset B(c,r).$$

We obtain an upper bound of $I_{B_k}$ solely in terms of
    $\int_{B(c,r)} \hat \Delta_{1}^{\zeta}(y)\,d\nu_c(y)$, as follows.
We change variables in the integrals $(\ref{equ:inegr})$ and use the fact that
the distortion of the inverse branch of $f^{m-1}_c$ mapping
$f^m_c(0)$ to $c$ is controlled on $B(x,r(x))$, for $x \in E_m$, by the technique of  shrinking neighbourhoods.

\begin{eqnarray}
I_{B_k}&\leq& 
\sum_{m=1}^k
\int_{E_m} \hat \Delta_{m}^{\zeta}(x)d
\nu_c(x)\label{equ:inegr}\\
&\lesssim & \sum_{m=1}^k\int_{ f_c^{-m+1}(E_m)} 
(\delta_m)^{-2\dconf(c)\zeta}~~\frac{\hat \Delta_{1}^{\zeta}(y)} { 
    |(f^{m-1})'(y)|^{-\dconf(c)(1-\zeta/2)} }
    \, d\nu_c(y)\nonumber\\
&\lesssim & \sum_{m=1}^k (\delta_m^{-4} 4^{(1-\zeta/2)m})^{\dconf(c)}
\int_{B(c,r)}{\hat \Delta_{1}^{\zeta}(y)}\; d\nu_c(y)\nonumber
\, ,
\end{eqnarray}
where the last line follows from $\zeta<2$ and $\sup_{x\in\J_c}|f_c'(x)|\leq 4$.

Provided $\zeta$ is sufficiently close to $2$,
\begin{equation}\label{equ:dense1}
\sum_{k=1}^{\infty}\gamma_{k}^{-\dconf(c)}
\sum_{m=1}^k (\delta_m^{-4} 4^{(1-\zeta/2)m})^{\dconf(c)}
\end{equation}
is uniformly bounded. Invoking $(\ref{equ:dense0})$, we see that
$\int_{B(c,r)} \rho_c(x)^{\zeta}d\nu_c$ can be  bounded by
\begin{equation}\label{equ:rp6}
\sum_{k=1}^{\infty} \gamma_{k}^{-\dconf(c)} I_{B_k}  
~\lesssim~ \int_{B(c,r)} 
{\hat \Delta_{1}^{\zeta}(x)}\, d\nu_c(x).
\end{equation}
This completes the case $r \leq R'/6$.

Let $B'_k:=B(f_c^k(0),R'/2)$ for all $k>1$. For $r \geq R'/6$ we have 
$$\int_{B(c,r)} \hat \Delta_{k}^{\zeta}(x)d\nu_c \leq
\int_{\J_c} \hat \Delta_{k}^{\zeta}(x)d\nu_c ~\lesssim ~
I_{B'_k} +(R')^{-\dconf(c)}~\lesssim~I_{B'_k}\, ,$$
because  $I_{B'_k}\geq  \int_{B'_k} d\nu_c \geq \nu_0$. 
Therefore, from (\ref{equ:dense0}) and \requ{rp6},
$$\int_{B(c,r)} \rho_c(x)^{\zeta}d\nu_c ~\lesssim~
\sum_{k=1}^{\infty} \gamma_{k}^{-\dconf(c)} I_{B'_k} 
~\lesssim~ \int_{B(c,R'/6)} 
{\hat \Delta_{1}^{\zeta}(x)}\, d\nu_c(x).$$
\end{proof}



Using the the estimate on $\nu_c(r)$ from Lemma \ref{lem:cl} and bounds \requ{integr1} and \requ{integr}, we obtain good control on $\sigma_c(r):=\sigma_c(B(c,r))$.

\begin{prop}\label{theo:regular}
	There exists $C>1$ such that 
	for  every $c\in {\A}$ and every $r_c^{1/2} \leq r \leq R'$, 
	$$\sigma_c(r)\leq C\, r^{1/2}.$$
\end{prop}
\begin{proof}
	
	It suffices to show the proposition for $\eps = c+2$ small. Then $\eta(\eps)$ defined by \requ{defeta} is greater than $1 - \frac{1}{8}$.
	
	Given $\zeta\in (1,2\eta(\eps))$, the H\"{o}lder inequality, together with inequality \requ{integr}, gives
	\begin{eqnarray}
		\sigma_c(r) &=& \int_{B(c,r)} \rho_c(x)  d\nu_c 
		\leq 
		\left(\int_{B(c,r)}
		\rho_c(x)^{\zeta} d\nu_c\right)^{1/\zeta}
		\left(\int_{B(c,r)}d\nu_c\right)^{1-1/\zeta}\nonumber\\
		&\lesssim &~(I_r)^{1/\zeta}\; \nu_c(r) ^{1-1/\zeta}\,,
		\label{equ:holder}
	\end{eqnarray}
where $I_r= \int_{B(c,r)} |x-c|^{-\dconf(c)\zeta/2}\,d\nu_c(x)$.

By inequality \requ{integr1}, 
\[
I_r ~\lesssim~ \frac{r^{1-\zeta/2}}{2-\zeta}+ \frac{r_c^{\eta(\eps)-\zeta/2}}{2\eta(\eps)-\zeta}\, .
\]
	Put $\zeta=3/2$. If $r > r_c^{1/2}$, then
\[
\frac{r^{1-\zeta/2}}{2-\zeta} \geq 2\, r_c^{1/8} \geq 2\, r_c^{\eta(\eps)-\zeta/2} \geq 8 \frac{r_c^{\eta(\eps)-\zeta/2}}{2\eta(\eps)-\zeta}\, ,
\] 
so $I_r ~\lesssim~ r^{1/4}$. Using the estimate of Lemma \ref{lem:cl} in \requ{holder}, we obtain
\[
\sigma_c(r) ~\lesssim~ r^{1/6}r^{\dconf(c)/3} ~\lesssim~ r^{1/2}.
\]
\end{proof}

\paragraph{Key ``diagonal" estimate.}
    \label{prop:sigreg}
Fix $t_0 \gg 1$. By Lemma~\ref{lem:rc}, $r_c < \eps^{2(1+t_0)}$, if $\eps$ is sufficiently small.  Proposition~\ref{theo:regular} implies that, for $r = a\, \eps$, $a \geq \eps^{t_0}$,
$$
\sigma_c(a\,\eps) 
	~\lesssim ~  a^{1/2} \, \sqrt{\eps} \;.
\label{equ:epspf}
$$

For the scales smaller than $r_c^{1/2}$ we have a slightly weaker estimate. 
Recall that $\tau(\eps)=\exp(-\kappa\sqrt{\log 1/\eps})$ is the function from Proposition~\ref{prop:realintro} and
$S$ the constant from Lemma~\ref{lem:c}.
\begin{prop}\label{theo:regular1} 
There exists $C>1$ such that for every function
$u(\eps)\geq S\tau(\eps)$ with $\lim_{\eps\rightarrow 0} u(\eps)=0$,  every
$c\in {\A}$ and $r\leq R'$, 
$$\sigma_c(r)\leq C\; u(\eps)^{-1/2}
\nu_c(B(c,r))^{(1-8u(\eps))/2}.$$
\end{prop}
\begin{proof}

We apply 
the H\"{o}lder inequality with the exponents $\zeta\in (1,2\eta(\eps))$, $\eta(\eps)$ is the function from Proposition~\ref{prop:eta}, and $\zeta'>0$, $1/\zeta+1/\zeta' =1$. 
Similarly to  (\ref{equ:holder}), 
\begin{eqnarray}
\sigma_c(r) &=& \int_{B(c,r)} \rho_c(x)  d\nu_c 
\leq \left(\int_{B(c,r)}d\nu_c\right)^{1/\zeta'}
\left(\int_{\J_c}
  \rho_c(x)^{\zeta} d\nu_c\right)^{1/\zeta}\nonumber\\
&\lesssim &~(2\eta(\eps)-\zeta)^{-1/\zeta}\; \nu_c(B(c,r)) ^{(1-1/\zeta)}\,.
\label{equ:holder1}
\end{eqnarray}
Put $\zeta=2\eta(\eps)-2u(\eps)$. We can assume that $S\tau(\eps)+u(\eps)<1/4$. 
Then 
$$1/2+4u(\eps)\geq 1/\zeta \geq 
1/2+u(\eps).$$
By $(\ref{equ:holder1})$,
\begin{eqnarray*}
\sigma_c(r) &\lesssim &  u(\eps)^{-1/\zeta}\;\nu_c(r)^{1-1/\zeta}\\
&\lesssim& {u(\eps)}^{-1/2}\; e^{-2u(\eps)\log u(\eps)}
\;
\; {\nu_c(r)^{(1-8u(\eps))/2}}\\
&\lesssim & 
 u(\eps)^{-1/2}\;\nu_c(r)^{(1-8u(\eps))/2}.
\end{eqnarray*}
\end{proof}
\begin{coro}\label{coro:log}
There exists a uniform constant $C>1$ such that for
all $r\leq R'$ and $c \in \A$,
$$\sigma_c(r)\leq  C \sqrt{|\log\eps|}\; r^{\frac{1}{2}\left(1-\frac{10}{|\log \eps|}\right)}.$$
\end{coro}
\begin{proof}
As $\HD(\J_c)=\dconf(c) \geq 1$, we may combine the estimate for $\sigma_c(r)$ with the upper bound in Lemma \ref{lem:c} to obtain
\begin{eqnarray}
\sigma_c(r) & \leq & C u(\eps)^{-1/2} r^{\HD(\J_c)(1 - S \tau(\eps))(1-8u(\eps))/2} \nonumber \\
&\leq & C u(\eps)^{-1/2} r^{1/2 - 5u(\eps)} \;. \nonumber
\end{eqnarray}
By choosing $u(\eps)=|\log(\eps)|^{-1}$ and $c_0 > -2$ such that for all $c\in \A \cap (-2,c_0]$, $S\tau(\eps) \leq u(\eps)$, we obtain the claim of Corollary~\ref{coro:log}.
\end{proof}

\paragraph{General upper bound for $\sigma_c$ near $c$.}Assuming that $\eps < e^{-50}$, Corollary~\ref{coro:log} implies for  $r = a\, \eps$, $a > 0$,
$$
\sigma_c(a\,\eps) 
   \lesssim   \max(a^{1/2}, a^{2/5}) \, \sqrt{\eps |\log\eps|} \;.
\label{page:epspf1}
$$

\section{Orbital estimates}
Recall the definition of $M$, the first return time of $0$ to $(q_c, -q_c)$. 
Furthermore, $\cU(z,r)$ is the set of times $n$ for which there is a level-$n$ univalent pullback of $B(f_c^n(z), r)$ containing $z$, see Definition~\ref{def:pullback}. 

\begin{lem} \label{lem:3nice}
    There exists $C > 1$ such that the following holds. 
    If
    $|z| \in  [\sqrt{10\eps}, 1/1000)$,
    let 
    $$e(z) =\min\{k \geq 2 \ :\ |f_c^{k}(z) - p_c| \geq 1/20\}.$$ 
    Then $2 \leq e(z) < M$. 
    There exist domains $\hat W \supset W \ni z$ such that 
    \begin{itemize}
        \item
            $f_c^{k}(\hat W) \subset  B(p_c, 1/2)$ for $k = 2, \ldots, e(z)$; 
        \item
            $\hat W$ is a level-$e(z)$ univalent pullback of 
            $B(f_c^{e(z)}(z), 1/1000);$
        \item
            if $V$ is a level-$k$ pullback, not necessarily univalent, of $B(f_c^k(z), R')$
            and $k > e(z)$, then $0 \notin V$;
        \item 
            $W$ is a level-$e(z)$ univalent pullback of 
            $B(f_c^{e(z)}(z), R'/2);$
        \item
            $B(z, |z|/C) \subset W \subset B(z, C|z|)$. 
    \end{itemize}
\end{lem}
\begin{proof}
Observe that
    $|f_c^{e(z)}(z)-p_c| < 1/5$ and that $|(f_c^k)'(-p_c)|=(2p_c)^k$.
    Let $z_n := f_c^n(z)$. 
    By choice of $R'$ in Subsection \ref{sec:const}, $B(z_k,4R') \subset B(p_c,1/10)$ for $k = 2, \ldots, e(z)-1$. 
    From formula \eqref{equ:pc}, $|c+p_c| < \eps$. Recall the discussion following \eqref{equ:pc}
    about the dynamics and distortion bounds for
 $g_c$, the inverse of the restriction of $f_c$ to $\{\Re (z) > 0\}$.
    It follows that $|f_c^k(c) - p_c| < |z_k - p_c|$ for $k = 0, \ldots, e(z)$, so $e(z) \leq M-1$. 

    As $g_c^{e(z)-1}(B(z_{e(z)}, 1/20))$ cannot contain $p_c$, 
    \begin{equation}\label{equ:wprime}
        \diam g_c^{e(z)-1}(B(z_{e(z)}, 1/1000)) < |f_c(z) + p_c|/4 < |f_c(z) - c|/2.
    \end{equation}
    Let $\hat W$ be the connected component of $f^{-2}_c(g_c^{e(z)-2}(B(z_{e(z)}, 1/1000)))$ containing $z$. 
    The third point follows from \eqref{equ:small}.

    The topological conditions of the lemma define $W$ uniquely. 
    The distortion on $W$ is uniformly bounded as $R'/2 <  1/1000$. 
    The derivative of $f_c^{e(z)-1}$ on $f_c(W)$ is comparable in modulus to $|f_c(z) + p_c|^{-1} \sim |f_c(z)-c|^{-1} = |z|^{-2}$. 
    Consequently,  $\diam (W)$ is comparable to $|z|$ and we obtain the final claim. 
\end{proof}

We fix  $K \geq 4$ satisfy $\xi^K < C^{-1}/4$, where $\xi, C$ are the constants of \eqref{equ:sqrt}. 

\begin{lem}\label{lem:threeblock}
    Given $z \in \J_c$ with $0 < |z| < 1$ and $n\geq K |\log |z||$, if $V \ni z$ is a domain with 
    $$
    f^n_c(V) = B(f^n_c(z), R'),$$
     then $0 \notin V$. 
\end{lem}
\begin{proof}
    By \eqref{equ:sqrt} and choice of $K$, 
    $\diam(V) \leq 4^{\log|z|} < |z|$.
\end{proof}

Assume that $x\in \J_c \sm \bigcup_{n \geq 0} f_c^{-n}(0)$
 and define the sets of integers 
\begin{eqnarray*}
    E_1 = E_1(x) & :=& \bigcup 
 \llbracket n+1 ,n + K \lceil |\log|x_n|| \rceil \rrbracket, \\
    E = E(x) & :=& \bigcup 
 \llbracket n+1 ,n + 3K \lceil |\log|x_n|| \rceil \rrbracket,
\end{eqnarray*}
where the unions are taken over $n$ for which $|x_n| \in (0, \sqrt{10\eps})$.

Let $E' = E'(x)$ denote the union of 
$
 \llbracket n+1 ,n + e(x_n)  \rrbracket$  over $n$ for which 
$|x_n| \in [\sqrt{10\eps}, 1/1000)$.
 Let $G = \N \setminus (E\cup E').$ We summarise useful properties in the following lemma. 

 \begin{lem} \label{lem:EGT2}
     The sets satisfy
     $G \subset \N \setminus(E_1 \cup E') \subset \cU(x, R')$. 
     Each interval $\ii{a,b} \subset E'$ has  $|b-a| < M-1$. 
     A maximal interval $\ii{a,b} \subset \N\setminus G$ with $b-a \geq M$ has 
     \begin{equation}\label{equ:EEprime}
     \# \ii{a,b} \cap E \geq \frac{b-a}2.
     \end{equation}

     If $n \in G$ and $n > 2K |\log \sqrt{10\eps}|$, there exists 
\begin{equation}
\label{equ:mnreal}
 m = m(n) \in \llbracket n-2K|\log\sqrt{10\eps}|, n-K|\log\sqrt{10\eps}| \rrbracket
\end{equation}
     such that $m \in \cU(x, R')$. 
 For $k \in \llbracket m, n \rrbracket$, 
\begin{equation}
\label{equ:complex}
 |x_k - c| \geq 10\eps.
\end{equation}
 \end{lem}
 \begin{proof}
     The first two statements follow from \eqref{equ:small} and Lemmas~\ref{lem:3nice} and~\ref{lem:threeblock}.
     For \eqref{equ:EEprime}, note that $\ii{a,b}$ necessarily intersects $E$ and therefore $b-a \geq 3M$, while maximal intervals in $E'$ have length bounded by $M-1$.

     To show~\eqref{equ:mnreal}, the interval of integers $ \llbracket n-2K|\log\sqrt{10\eps}|, n-K|\log\sqrt{10\eps}| \rrbracket$ has length at least $M$ and is contained in $\N\setminus E_1$ and hence contains a number $m$  in $\N \sm(E_1 \cup E')$; necessarily $m \in \cU(x, R')$. 
     The final statement follows from $\ii{m,n} \cap E_1 = \emptyset$. 
 \end{proof}

 %

 %

\paragraph{Estimates based on ergodicity.}\label{sec:erg} 


Let $p_0$ be maximal with $e^{-p_0} \geq 10\eps$. 
Set $\chi_p := \chi_{B(c, e^{-p})}$, the characteristic (or \emph{indicator}) function of the ball $B(c, e^{-p})$. 

Let  $X$ denote the set of points $x \in \J_c \sm \bigcup_{n \geq 0} f_c^{-n}(0)$ 
which satisfy,
for each function $h$ of the form
\begin{itemize}
    \item
        $\chi_p$, $p =1, \ldots, p_0$,
    \item
        $h_p = 3K p \chi_p$, $p = 1, \ldots, 2p_0$,
    \item
        and 
        $h_\infty = \sum_{p=2p_0}^\infty 3Kp \chi_p$,
\end{itemize}
\begin{equation}\label{equ:hfreq}
\lim_{N \ra \infty} \frac 1N \sum_{k=0}^{N-1} h(f_c^k(x)) = \int h \, d\sigma_c.
\end{equation}
By Birkhoff's ergodic theorem, $\sigma_c(X) = 1$. 

Fix a point $x = x_0 \in X$ and consider the sets of integers $E_1 = E_1(x), E=E(x), E'= E'(x)$ as before. 
The asymptotic upper density of a set of integers $Q$ in $\N$ is defined by
\begin{equation}
\label{def:density}
d(Q):=\limsup_{N\rightarrow+\infty}\frac{1}{N}\#\{Q\cap [1,N]\}.
\end{equation}
We compute upper bounds for the asymptotic upper density  of $E$ 
 using the definition of $p_0$ and inequalities \eqref{equ:eps} and \requ{eps1}:
\begin{eqnarray}
\label{equ:eps2}
    d(E) & \leq & 3K \sum_{p=p_0}^{2p_0} p C_1 e^{-p/2} + 3K \sum_{p=2p_0+1}^\infty p C_1 e^{-(p-p_0)2/5} \sqrt{\eps|\log \eps|}
\\
    & \leq &  6 KC_1p_0 \left( \sqrt{e^{-p_0}} + e^{-2p_0/5}\sqrt{\eps|\log \eps|} \sum_{p \ge 1} pe^{-2p/5} \right)  \nonumber \\
    & \leq & C_2\eps^{1/2}|\log\eps|, 
\end{eqnarray}
where $C_1, C_2 >0$ are uniform constants; in particular, while $E$ depends on $x\in X$, the bound \eqref{equ:eps2} holds for all $E$ and $\eps$ under consideration. Subsequent bounds will similarly work for all considered $\eps$ and all $x \in X$.


\paragraph{Density of iterates versus density of scales.}\label{sec:dens} 

We want to  translate the density of $G= \N \setminus (E\cup E')$ in $\N$ into the density  of the corresponding scales at $x \in X$. 
Given $n \in G$,
let $n'>n$ denote the next smallest element of $G$. 
If $n'-n = 1$ or $n'-n \geq M$, let $j_n = 1$. Otherwise, $n+1 \in E'$ and $|x_n| \in [\sqrt{10\eps}, 1/1000)$. Let $j_n$ be minimal with 
$R'2^{-j_n} \leq |x_n|$. 
We consider 
$$\cN = \{ (n,j) : n \in G, 1\leq j \leq j_n\}.$$

For $n \in G$, denote by $r_{n,1}$ the maximal radius with $f_c^n(B(x,r_{n,1})) \subset B(x_n, R'/2)$. 
For $(n,j) \in \cN$, let $r_{n,j} = 2^{-j+1}r_{n,1}$.
By bounded distortion, if $n'-n <M$, 
$$\diam(f_c^n(B(x, r_{n,j_n}))) \sim |x_n|.$$ 
By Lemma~\ref{lem:3nice}, 
 $r_{n,j_n}$ is comparable to $r_{n',1}$ if $n'-n < M$.
 Therefore we can fix $k_*\geq 1$ so that 
$$2^{k_*}r_{n', 1} > r_{n, j_n}$$
for all such pairs $n, n'$, $n'-n < M$. 

Let $Q$ denote the set of integers $k$ for which 
there exists $(n,j) \in \cN$ with $r_{n,j} \ss I_k$, 
where $I_k=(2^{-(k+k_*)},2^{-k}]$. 

\def\sc{\pi} 
Define a function $\sc: G\mapsto \N$,
$\sc(n)$ is the  smallest integer $k$ such that
$r_{n,1}\in I_k$. 
By \eqref{equ:sqrt}, for large $n$, $n ~\lesssim~ \sc(n)$.

Suppose $n, n' \in G$. 
Then $\sc(n') - \sc(n) \leq C_3 (n'-n)$ (since the set $W \ni x_n$ mapped univalently by $f^{n'-n}_c$ to $B(x_{n'}, R'/2)$ contains $B(x_n, 4^{-n'+n} R'/2)$). We shall apply this estimate whenever
 $\llbracket n +1, n'-1 \rrbracket$ contains a component of $E$, 
recalling~\eqref{equ:EEprime}. 
On the other hand, if $\llbracket n, n' \rrbracket$ is a subset of $\N \setminus E$, 
then $\llbracket \sc(n) , \sc(n') \rrbracket \subset Q$, by choice of $k_*$. 

Therefore $d(\N \setminus Q) ~\lesssim~ d(E) ~\lesssim~ \eps^{1/2}|\log \eps|$. 

We need additional estimates on the density  of $G_j := \{n \in G : (n,j) \in \cN\}$, when $j \geq 2$. 
If $n \in G_j$ then, crudely,  $|x_n| < C_4 2^{-j}$. This latter condition happens with frequency bounded by $\sigma_c(B(0, C_4 2^{-j})) ~\lesssim~ 2^{-j}$. Hence $d(G_j) ~\lesssim~ 2^{-j}$ and, as $n ~\lesssim~ \sc(n)$, 
\begin{equation}
    \label{equ:pigj}
d(\sc(G_j)) ~\lesssim~ 2^{-j}. 
\end{equation}

\section{Proof of Theorem \ref{theo:hausTop}}\label{sec:upp} 
Recall Definition~\ref{defi:beta} of $\beta(x,r)$-numbers for 
$K \se \R^d$ with $d\geq 2$.
Observe that in general, for $x \in K$ and $0<r<r'$,
\begin{equation}
\label{equ:betaScale}
\be_K(x, r) \leq \frac{r'}{r} \be_K(x,r')\,.
\end{equation}
Hence 
$$
\be_K(x, 2^{-k-k_*}) \leq 2^{k_*} \inf_{r \in I_k} \be_K(x,r)\,.
$$

\paragraph{Almost flat sets.}
Let $0 = \be_0 < \be_1 < \ldots < \be_n = 1$ and $\mathcal F = \{(d_i, \beta_i)\, : \, d_i \in [0,1] \text{ for all }i=1,\ldots,n\}$. 
A set $K\subset \R^d$ is {\em almost flat} at a point $x\in K$ with respect to $\mathcal F$ if there exist a partition $\sqcup_{i \in \ii{0,n}}Q_i=\N$ such that for all $i \in \ii{0,n}$ and $m \in Q_i$
\[
\beta(x, 2^{-m}) \leq \beta_i \text{ and } d(Q_i) \leq d_i \,,
\]
recalling the definition (\ref{def:density}) of the upper density $d(Q)$ of a subset $Q$ of $\N$. The following fact is a direct corollary of \cite[Theorem 2]{gjm}.

\begin{fact}\label{thminvmean}
Suppose that the set $K\subset \R^d$, $d\geq 2$, is almost flat at every point $x\in K$ with respect to a given family $\mathcal F$ as above.
Then, 
$$\HD(K) \leq 1  + C_d'\sum_{i=1}^n d_i\, \beta_i^2\,, $$
where for each $d \geq 2$, $C_d' > 0$ is a universal constant.
\end{fact}

\paragraph{Geometric estimates.}
The upper bound of Theorem~\ref{theo:hausTop} will follow from the estimates of Section~\ref{sec:erg} and from Fact~\ref{thminvmean} about the Hausdorff dimensions of almost flat sets. An initial geometric estimate comes from  \cite[Proposition~2]{bvz}
stated as Fact~\ref{fa:bvz}.
\begin{fact}\label{fa:bvz}
If a quadratic Julia set $\J_c$ is connected and $c \ne -2$ then
$$\J_c\subset E_4:=\{z\in \C: |z-c|+|z+c|<4\}.$$
\end{fact} 
Let $\eps=c+2$, $c>-2$,  be close to $0$. Then the Julia set $\J_c$ is contained in the horizontal strip $B(\R, 2\sqrt{\eps})$.
{\cred 
\begin{lem}\label{lem:wedge}
There exists a constant $C>0$ such that for all $c \in \A$,
\[
\J_c \subset \brs{z \in \C\ :\ |z| \leq p_c \text{ and } |\Im(z)| \leq C\sqrt \eps (p_c - |\Re(z)|)}\,.
\]
\end{lem}
\begin{proof}
    Consider the rectangle 
    $$
    H_\eps := \{z : \Re(z) < 5/4,\, \Im(z) < 2\sqrt \eps\}.$$
    It contains $\J_c \cap \{z : \Re(z) < 5/4\}$. 

Let $\J_c^+:=\{z \in \J_c \,:\, \Re(z) \ge 0\}$. 
Using the definition of the map $g_c$ on page \pageref{def:gc}, we can see that 
\[
\J_c^+ \subset \{p_0\} \cup \bigcup_{n \ge 0} g_c^n(\ol{H_\eps}) \,.
\]
Then $\J_c^+$ is included in a half cone at $p_c$ with aperture comparable to $\sqrt \eps$. As $\A \se \R$, $\J_c$ is symmetric w.r.t.\ both axes.
\end{proof}

Combining Lemma~\ref{lem:wedge} with \eqref{equ:complex} and formulas \requ{pc}, we obtain the following.
\begin{coro}\label{coro:G}
    For $c \in \A$ sufficiently close to $-2$ and $x \in X$, if $n \in G$ and $m(n)$ is given by \eqref{equ:mnreal}, then
\[
\Re(x_k) \geq c + \eps\, ,
\]
    for all $k \in \llbracket m(n), n \rrbracket$. 
\end{coro}

In order to recover estimates of $\beta$ numbers at the large scale via univalent pullbacks, we will employ the following version of Koebe's Theorem.
\begin{lem}\label{lem:koebe}
Let $g : \D \ra \C$ be univalent satisfying $g(0)=0$ and $g'(0)=1$. Then for all $z \in \D$,
\[
|g(z) - z| \leq |z|^2 \frac{2 - |z|}{(1-|z|)^2}.
\]
\end{lem}
\begin{proof}
Let $g(z)=z+\sum_{k \geq 2}a_k z^k$. By the Bieberbach conjecture, proven by de Branges \cite{bra85}, we know that for all $k \geq 2$, $|a_n| \leq k$. Thus it is enough to compute a bound for 
\[
\left| \sum_{k \geq 2}a_k z^k \right| \leq \sum_{k \geq 2}k |z|^k = |z|^2 \frac{2 - |z|}{(1-|z|)^2}\, .
\]
\end{proof}
\begin{coro}\label{coro:gaffine}
    Let $g : \D \ra \C$ be univalent, $g(0)=0$, $g'(0)=1$ and $|z| = r \leq 1/6$. Then
\[
|g(z) - z| \leq 3r^2\,.
\]
\end{coro}

\begin{coro}\label{coro:gstrip}
    Let $g : \D \ra \C$ be univalent, $g(0) = 0$ and $g'(0)=1$, and let $r \leq 1/6$. Let $Y \subset B(0,r)$ and $L$ a line through 0 with $g(Y) \subset B(L,\rho)$, for some $\rho >0$. 
    Then $Y \subset B(L, \rho + 3r^2)$. 
\end{coro}
\begin{proof}
    Suppose $z \in B(0,r)$ with $\dist{z, L} \geq \rho+3r^2$. By  Corollary~\ref{coro:gaffine} and the triangle inequality, $\dist{g(z), L} \geq \rho+3r^2 - 3r^2.$
\end{proof}

%
%
We set $\beta(z,r) := \beta_{\J_c}(z,r).$
The real line is a good comparator when estimating $\beta$-numbers for $c$ near $-2$. 
For $x \in X$, we have a corresponding set of neighbourhoods  of $x$. 

\begin{lem}\label{lem:scary}
    There exists $C >1$ such that, given
     $(n,j) \in \cN$, 
     $$\beta(x, r_{n,j}) \leq C 2^j \eps^{1/2}.$$ 
\end{lem}
\begin{proof}
    It is enough to show this when $j=1$, and then apply \eqref{equ:betaScale}. 

    Given $n \in G$, let $m=m(n) \in \cU(x,R')$ be given by \eqref{equ:mnreal}. 
    Each $x_k$, $k \in \llbracket m(n) ,n \rrbracket$ has real part at least $c + \eps$ by Corollary~\ref{coro:G}. 
    The imaginary part of $x_n$ is bounded by $2\eps^{1/2}$, so $x_n$ is very near the centre of the large-scale line segment $B(x_n,R'/2) \cap \R$. 
    Consequently (as $m > 2$ and $n \in \cU(x,R')$), $B(x_n, R'/2) \cap \R \subset (c, \infty)$.

    Consider the univalent pullback $W_k$ of $W_0 := B(x_n,R'/2)$ by $f_c^k$ to $x_{n-k}$.
    Note that $c \notin W_k$.
    As $W_0 \cap \R \subset (c, \infty)$, 
    $f_c(W_1 \cap \R) = W_0\cap \R$ and $W_1 \cap \R$ is a line segment. 
    By bounded distortion,  $x_{n-1}$ lies in the vertical strip with real part $W_1 \cap \R$, so $W_1\cap \R \subset (c,\infty)$. 
    
    Repeating this argument for $k=2, \ldots,n-m$, we obtain that
    $$f_c^{n-m}
    (W_{n-m} \cap \R) 
    = B(x_n,R'/2) \cap \R.$$ 
    Let $W:= W_{n-m}$. 
    We deduce that $J_c \cap W$ is very close to the real axis:  by bounded distortion, if $w \in \J_c \cap W$, then 
    \begin{equation}
        \label{equ:diamw}
        |\Im(w)| ~\lesssim~  \frac{ \diam(W)}{R'} \sup_{z\in\J_c}
        |\Im(z)| ~\lesssim~ \diam(W)
        \, \eps^{1/2}.
    \end{equation}

    By the definitions, $f_c^m(B(x, r_{n,1})) \subset W$.
    Since $n-m > K|\log\eps|/2$, 
    $$\diam(W) < \sqrt{\eps}.$$ 
    By choice of $m$, there is a level-$m$ univalent pullback  $\hat W \ni x$ of $B(x_{m}, R')$. 
    By bounded distortion, for a uniform constant $C_1>1$, 
    $$B(x,r_{n,1} \diam(W)^{-1}/C_1) \subset \hat W.$$
    This gives us the modulus needed to apply Corollary~\ref{coro:gstrip}.

    From~\eqref{equ:diamw}, for $z \in f_c^m(B(x,r_{n,1})) \cap \J_c$, the distance from $z$ to the horizontal line passing through $x_m$ is bounded, $\lesssim~ \diam(W) \eps^{1/2}$. 
    Applying Corollary~\ref{coro:gstrip} (and some affine transformations, with $r=C_1\diam(W)$), 
    $$\beta(x, r_{n,1}) ~\lesssim~ \eps^{1/2} + 3\eps C_1^2 ~\lesssim~ \eps^{1/2} .$$
\end{proof}
}

Recall our set $Q \subset \N$ of controlled scales. We have that, for some uniform constant $C_2>0$, 
\[
d(\N \setminus Q) \leq C_2  \eps^{1/2} |\log \eps|\,.
\] 

For each scale  $k \in Q$, we choose some $r_{n,j} \in I_k$ and  set $\zeta(k) := j$. 
Then $d(\{k : \zeta(k) = j\}) \leq k_* d(\sc(G_j)).$
By Lemma~\ref{lem:scary}, we obtain an associated beta number $\beta(x, r_{n,j}) \leq C 2^j \eps^{1/2}$.
Hence 
$$\beta(x, 2^{-k-k_*}) \leq 2^{k_*} C 2^j \eps^{1/2} = : C_3 2^j \eps^{1/2}.$$ 
We set
$$\be_j := \min(1, C_3\, 2^j\eps^{1/2}).
$$
We associate $\beta=1$ to $\N \sm Q$ (shifted by $k_*$), set $d_1 = 1$ and 
\[
    d_j:=k_*\, d(\sc(G_j)) ~\lesssim~ 2^{-j}\,.
\]  

Thus $d_j\, \be_j^2 \leq C_4 2^{j} \eps$, where $C_4 > 0$ is a uniform constant.
Writing $j_* := \max \{j_n : n \in G\}$, then $2^{-j_*} \sim \sqrt{\eps}$ and, summing over $j$ we obtain
$$\sum_{j=1}^{j_*} d_j\, \beta_j^2 \leq C_4\, 2^{j_*+1} \eps ~\lesssim~ \eps^{1/2}\,.$$

The upper bound on $d(\N \sm Q)$ and applying Fact \ref{thminvmean} 
 completes the proof of Theorem \ref{theo:hausTop}.

\section{Induced Cantor repellers and the lower bound}\label{sec:rep}
We construct an induced Cantor repeller in Proposition~\ref{prop:ext}. 
In Proposition~\ref{prop:haus}, we determine a lower bound on the Hausdorff dimension of its Julia set. As
its Julia set is contained in the Julia set of $f_c$, this will prove Theorem~\ref{theo:hausBot}. 

\subsection{Preliminaries}

\paragraph{Cantor repellers and inducing} 
\begin{defi}\label{defi:rep}
Suppose that $D_1,\dots, D_n$ is a collection of open and non-degenerate topological disks with pairwise disjoint closures compactly contained in a topological disk $D \subset \C$.
A map $\vp:\bigcup_{i=1}^n D_i\mapsto D$ which is biholomorphic onto $D$ on every $D_i$, $1\leq i\leq n$, is called   a 
{\em Cantor repeller}.
\end{defi}
If $\vp$ preserves the real line
and each branch domain $D_i$ is symmetric with respect to $\R$ then $\vp$ is 
a {\em real} Cantor repeller. 
With respect to a map $f$, if there are integers $n_i$ such that $\vp_{|D_i} = f^{n_i}$, we say that $\vp$ is \emph{induced} (by $f$).

Every $f_c(z)=z^2+c$, $c\in [-2,0)$ has two fixed points $p, q \in [-2,2]$, $0 < -q < p$, and is unimodal on $[-p,p]$. 
The non-empty interval $U=(q,-q)$
is called a {\em fundamental inducing interval}. $U$ is a \emph{regularly returning} set, that is, $\forall n>0,~ f_c^n(\partial U)\cap U =\emptyset$. 

Let  $\phi = \phi_c$  be the first return map to $U$ (under the unimodal map $f:=f_c$, restricted to $\R$), defined on 
$${\calD}_\phi:= \{x \in U:\exists n >0, f^n(x)\in U\}$$ by the formula
$\phi(x) :=f^{n(x)}(x)$ where $n(x) := \min\{n >0 :f^n(x)\in U\}$. 
As $U$ is regularly returning,  the function $n(x)$ is continuous and locally constant on ${\calD}_\phi$. 
For $c$ close to $-2$, the set $\{x \in U: n(x) = 2\}$ has two connected components, $d_q$ adjacent to $q$ and $d_{-q}$ adjacent to $-q$. 
We define another regularly returning interval 
\begin{equation} \label{eqn:Vdef}
V = U\setminus \overline{d_q\cup d_{-q}}.
\end{equation}

 For a Borel subset $X \subset \R$, we denote by $|X|$ its Lebesgue measure. Given an interval $W \subset \R$, we denote by $\D_W$ the disc in $\C$ with diameter $W$. 
\begin{prop}\label{prop:ext}
There exist $C, \alpha>0$ and $c^*>-2$ such that for every $c\in (-2,c^*)$, there is a Cantor repeller $\varphi:{\calD} \mapsto \D_V$ induced by $f_c$
with range $\D_V$, with the following properties, 
$\eps=c+2$,
\begin{itemize}
    \item
each branch of $\varphi$ is extensible as a univalent map  onto $\D_{U}$,
    \item
the map $\varphi$ is defined on ${\calD}\subset \D_V$,
        \begin{equation}\label{largeintersection}
            |{\calD}\cap \R| \geq |V|(1 - C\eps^{3/4}),
        \end{equation}
\item  there is exactly one component $W$ of 
${\calD}$ for which $\phi$ restricted to $\calD\setminus W$ is a real Cantor repeller, and
\begin{equation}\label{eqnWest}
    \diam {W} \geq C^{-1}\sqrt{\eps},
\end{equation}
 \item  for every  $x \in {\calD}$,  $2 \leq |\varphi'(x)|$, 
 \item for all $t>0$, 
     $$\left|\{x \in \R  : |\varphi'(x)| > e^t\}\right| < Ce^{-\alpha t}.$$
\end{itemize}
\end{prop}
{\cblue
By \emph{extensibility} above, we mean that the  branch of $\phi$ is the restriction of a biholomorphic map between a larger domain and $\D_{U}$. By the Koebe distortion theorem, the distortion of any branch of $\varphi$ is bounded by a constant depending only on the modulus of $\D_U\setminus \D_V$. 
}
\begin{coro}\label{corTails}
    With $\vp, \alpha$ as above, the number of branches $\zeta$ for which $\inf|\zeta'(z)| \in [e^n, e^{n+1}]$ is bounded, for some uniform constant $C'$, by 
    $C'e^{n(1-\alpha)}$. 
\end{coro}
\begin{proof}
    Given $n \geq 1$ and such a branch $\zeta$, $\sup|\zeta'(z)| \leq e^{n+C_1}$ for some uniform constant $C_1$. 
    If there are $N$ such branches, the Lebesgue measure of the real line intersected with the union of the domains  is at least $Ne^{-n-C_1}$ but is bounded by $Ce^{-\alpha n}$.
    Hence
    $$ 
    N \leq C e^{-\alpha n + n + C_1}.$$
\end{proof}

The estimates are essentially real and require us to study the dynamics of $f_c : \R \to \R$. 
To obtain the estimates, we carry out some fundamental inducing steps to canonically-defined  \emph{box mappings}. The reader familiar with such inducing schemes can skip to the next key estimate, Lemma~\ref{lem:start}.

\paragraph{Box mappings.}
Consider a finite sequence of compactly nested open intervals around a point $ 0 \in b_0\subset b_1\dots \subset b_k$.  Let $\phi: \calD \mapsto \R$ be a real-valued 
$C^1$ map defined on some open and bounded set ${\calD} \subset b_k \subset \R$ satisfying the follwing:
\begin{itemize}
    \item
        $\phi$ has at most one local extremum which, if it exists, is at $0$;
    \item
        if $0 \in \calD$, then $b_0$ is a connected component of $\calD$;
\item for every $i=0,\cdots,k$, we have that 
$\partial b_i \cap {\calD}=\emptyset$;
\item for every connected component $d$ of $\calD$ 
there exists $0\leq i\leq k$ so that $\phi: d  
\mapsto b_i$ is proper.
\end{itemize}
The map $\phi$ is a {\em box mapping} and the intervals $b_i$ are called boxes. 
If $\phi$ has a local extremum, it has a \emph{central branch} 
$\psi:=\phi_{|b_0}$ and $b_0$ is called the \emph{central domain}. All other branches $\zeta_d := \phi_{|d}$  are monotone. 



A box map $\phi$ is  {\em induced} by a unimodal map 
$f_c(x)=x^2+c$ if each branch of $\phi$ coincides on its domain with an iterate of $f_c$. 
We shall construct  box mappings with up to four boxes $b \subset Z \subset V \subset U$, where $Z$ is an interval to be defined and $b$ is the central domain, should such exist. 


\subsection{Exponential tails}
We say that a map $g$ has \emph{exponential tails} if there are $C, \theta >0$ such that  $|\{|g'|\geq e^t\}|  < Ce^{-t\theta}$ for all $t\geq 0$. 
A family of maps $(g_\eps)_{\eps \in A}$ has \emph{uniform exponential tails} if all $g_\eps$ have exponential tails, with constants $C,\theta$ independent of $\eps \in A$. 
\begin{lem} \label{lemTailscrit}
    Let $I \subset [-2,2]$ be an interval containing the domains of maps $h$ with uniform exponential tails.  Given a 
     compact family of non-zero real polynomial maps $f$ defined on $\R$ 
    and a family  of expanding  diffeomorphisms $g : Y \to I$ with uniformly bounded distortion,
    \begin{itemize}
        \item
            $h \circ f$ has uniform exponential tails;
        \item
            $h \circ g$ has uniform exponential tails.
    \end{itemize}
\end{lem}
\begin{proof}
    There are $C, \theta >0$ such that  $|\{|h'|\geq e^t\}|  < Ce^{-t\theta}$ for all $t\geq 0$ and $C_1, d$ such that $|f'| < C_1$ on $f^{-1}([-2,2])$ and $|f^{-1}(A)| \leq C_1 |A|^{1/d}$ for every Borel subset $A \subset I$. 
    If $|(h\circ f)'(x)| \geq e^t$ then $|h'(f(x))| \geq C_1^{-1}e^t.$
    Hence
    \begin{eqnarray*}
        \left|\{x : |(h \circ f)'(x)| \geq e^t\}\right| & \leq & \left|f^{-1}\left(\{ y : |h'(y)| \geq C_1^{-1}e^t\}\right)\right| \\
        & \leq
        & C_1C^{1/d} e^{-\theta(t-\log C_1)/d},
    \end{eqnarray*}
    from which the first estimate follows. 

    The second estimate is straightforward. 
\end{proof}
In particular, pulling back via quadratic maps does not destroy exponential tails. 
The following lemma will be used to show that the composition of well-behaved maps with 
uniform exponential tails will have uniform exponential tails. 
\begin{lem}\label{lemTails}
    Given $K, C, \theta >0$ there are $C', \kappa>0$ for which the following holds. 
    Let $X$ be an open interval and $H : X \to [1,\infty)$. Let $g$ be a function, defined on an open set $Y \subset \R$, 
    which maps each branch domain of $g$ diffeomorphically onto $X$ with distortion bounded by $K$ and with $|g'|\geq 1$. 
    Let $$B_t = \{x : H(x) > e^t \};
    \quad
     A_s = \{x : |g'(x)| > e^s \}.$$
    If
    $$
    |B_t| \leq Ce^{-t\theta} |X| \text{ and }
    |A_s| \leq Ce^{-s\theta}$$
    for all $s,t\geq 0$,
    then
    $$
    \left|\{x \in Y : |H(g(x))g'(x)| > e^t \}\right| \leq C'e^{-t\kappa}$$
    for all $t\geq 0$. 
\end{lem}
    \begin{proof}
        Let $W_s$ denote the union of all branches of $g$ which contain a point $x$ with $|g'(x)| \in [e^s, e^{s+1})$. Crudely, $W_s \subset A_{s-K}$ and $W_s \cap A_{s+K} = \emptyset$. 
    Then 
    \begin{align*}
        \left|\{x \in W_s : |H(g(x))g'(x)| > e^t \}\right| & \leq \left|A_{s-K}\right|K\frac{|B_{t-s-K}|}{|X| }\\
                                                    &
    \leq C^2K e^{2K\theta}e^{-s\theta}e^{-(t-s)\theta} \\
    & = C^2K e^{2K\theta} e^{-t\theta}.
\end{align*}
    Now sum over integers $s \geq 0$   to obtain
    $$
    \left|\{x : |H(g(x))g'(x)| > e^t \}\right| \leq C^2Ke^{2K\theta}e^{-t\theta}(t+1) + |A_{t-K}|,$$
    from which the result follows. 
\end{proof}

\subsection{Initial inducing.}
We shall construct  successive box mappings with up to four boxes $b \subset Z \subset V \subset U$, where $Z$ is an interval to be defined and $b$ is the central domain of the box mapping, should such exist. 
We shall call  branches mapping monotonically onto $V$ or $U$ \emph{long} and other branches will be called \emph{short}.

\begin{defi}
    A diffeomorphism $\zeta : d \to W$ between intervals $d$ and $W$ is said to \emph{extend} (to map) \emph{over} an interval $\hat W \supset W$ if there are $\hat d \supset d$ and an analytic diffeomorphism $\hat \zeta : \hat d \to \hat W$ 
    whose restriction to $d$ coincides with $\zeta$. We say $\zeta$ is \emph{extensible} over $\hat W$ and call $\hat d$ the \emph{extension domain}.
\end{defi}
Note that if $\zeta$ is a restriction of $f_c^j$ for some $j$, then so is $\hat \zeta$. 

We consider the first return map $\phi:{\calD}_\phi\mapsto U$ of $f_c(z)=z^2+c$ to the fundamental inducing interval $U = (q,-q)$. 
If $\phi$ has no central domain, we can just take $\phi_*$ to be the first return map to $V$ and skip to Section~\ref{sec:cantorconstr}.
Henceforth, to avoid unnecessary caveats, we assume that the central branch $\psi : b \to U$ exists. 
  
We denote by $\zeta_l$ and $\zeta_r$ the two branches adjacent to the central branch; these branches are monotone. Denote by $Z$ the smallest interval containing the domains of $\zeta_l, \psi, \zeta_r$. 
Put  $\hat{U}:=(-\gamma,\gamma)$,
where $f_c(\gamma)=-q$ 
{\cblue
(this $\hat U$ is the same interval as $\hat A$ of Section~\ref{sec:yoc})}.
Components of the following fact are  well-known or follow by elementary arguments, noting that $|c+p| \sim \eps = c+2$.

\begin{fact}\label{fa:fund}
There exist $C>0$ and $c^*>-2$ such that for every $c\in (-2,c^*)$, the first return map $\phi$ for $f_c$ to $U$
is a
 box mapping with boxes $b, U$. Every monotone branch, except possibly $\zeta_l$ and $\zeta_r$, extends  over $\hat{U}$. Additionally, 
\begin{description}
\item{\rm (i)} $|U|\leq C\dist{U,\partial \hat{U}} $, 
\item{\rm (ii)} 
$|Z|\leq C \sqrt{\varepsilon}$,
\item{\rm (iii)} $\phi$ has only finitely many  monotone branches and for every  $x \in {\calD}_\phi\setminus Z$,  
$$
 3 \leq |\phi'(x)| \leq
C/\sqrt{\varepsilon},
$$
\item{\rm (iv)}
    each branch $\zeta_l, \psi, \zeta_r$, with domain $d$ say, can be represented as $f^i\circ f$ and there is an interval $W \supset f(d)$ such that $f^i:W\to \hat{U}$ is a diffeomorphism onto $\hat{U}$ with $\inf_{W}|(f^i)'| \geq C^{-1}\eps^{-1}$, 
\item{\rm (v)}
    $\phi$ has uniform exponential tails. 
\end{description}
\end{fact}
{\cblue
Indeed, (i) is trivial. 

Let $r\in (0,-q)$ satisfy $\hat U = (f(r), -f(r))$, and let $g = f_{|[r,p]}.$ 
The intervals $g^{-k}([r,p])$ decrease geometrically so $g^k$ has uniformly (in $k$ and in $c$) bounded distortion on its domain, noting that $r$ is far from $0$ for $c$ close to $-2$. 
Then 
\begin{equation}
    \label{eqgeomacc}
    \dist{g^{-k}(\hat U), p} \sim |g^{-k}(\hat U)| \sim |g^{-k}(U)|.
\end{equation}
Together with $|c+p| \sim \eps$, one readily deduces (ii), (iv) and the upper bound of (iii). The lower bound of (iii) follows from the estimate, see \cite[Page 5]{YocJak}, 
$$ 
|(f_{-2}^i)'(x)| = 2^n \frac{h(x)}{h(f_{-2}^i(x))} \in [2^i \sqrt{3/4}, 2^i\sqrt{4/3}]$$
provided that $x, f_{-2}^{i}(x) \in (-1,1) = U$; 
$h$ denotes the conjugacy between $f_{-2}$ and the full tent map. For $i$ small, the estimate transfers to other $c$ by continuity; for large return time $i$, the estimates on distortion  and on $|g^{-k}(U)|$ kick in, with $k+2 = i$. 

It remains to show (v).
A point $x \in [-q,p)$ has initial orbit satisfying 
$$
x > f_c(x) > \cdots > f_c^k(x) \in [q,-q],$$
where $k$ is the first entry time to $[q,-q] = \overline{U}.$
Let $s >2$ and let $X_s$ denote the set of points $x$ in $(-q,p)$ where the first entry of $x$ to 
$U$ happens with derivative greater than $s$ but the same is not true of $f_c(x)$. Then $|X_s| \leq |U|/s$. Moreover, $X= \cup_k g^{-k}(X_s)$ has measure bounded by $2|U|/s$ (estimated via a geometric series). 
If $x \in U$ and 
$|\phi'(x)| > 2s$,
then $-f_c(x) \in X$. 
Hence the set of points in $U$ with 
$|\phi'(x)| > 2s$
is contained in a set of measure $\sqrt{2|U|/s}$, showing (v). 
}


\subsection{Inducing}
In the inducing process, we shall pre-emptively use \emph{boundary refinement}, applying the following map $h_V$, 
to avoid creating long non-extensible branches in the pull-back step. The technique of the boundary refinement was proposed in~\cite{jakswia} in the quest to prove 
the so-called starting condition for unimodal maps~\cite{book}.
Denote by $h_V$ the first entry map for $f_c$ from $U$ to the interval $V$ (defined in~\eqref{eqn:Vdef}). It is a box mapping with two boxes $V, U$, is defined almost everywhere on $U$, coincides with the identity map on $V$,  and all its branches are long, diffeomorphic onto $V$ and extensible over $U$ with extension domains contained in $U$. We remark that $h_V$ has exponential tails, as $f_c$ restricted to $[-p,p] \setminus V$ is uniformly expanding. 


We now describe a process which transforms the first return map $\phi$ to $U$ into a new box mapping ${\phi_*}$ with up to three boxes $b_* \subset Z \subset V$. 
\subparagraph{Postcritical  filling.}

Let
$${\phi_0}(x)=\left\{
\begin{array}{ll}
x & \text{ if } x\in Z,\\
    \phi(x) &  \text{ if } x \in U\setminus Z.
\end{array}
\right.$$
We construct $\phi_j$ algorithmically for $j =1,2,\ldots$  
and denote the resulting limit map by $\phi_\infty$.
If $\psi(0)$  does not belong to the domain of a long branch of $\phi_{j-1}$, let $\phi_j = \phi_{j-1}$, so $\phi_\infty = \phi_{j-1}$ and the process stops. 

Otherwise, $\psi(0)$ belongs to the domain $d_{P,j}$ of a long branch $\zeta_{P,j}$ of $\phi_{j-1}$. We modify $\phi_{j-1}$ on $d_{P,j}$ to obtain $\phi_j$. 
Set 
$${\phi_j}(x)=\left\{
\begin{array}{ll}
    \phi_{j-1}(x) & \text{ if } x\in U\setminus d_{P,j},\\
    \phi_0\circ \phi_{j-1}(x) & \text{ if } x \in d_{P,j}.
\end{array}
\right.$$
If $x \notin Z$ then $|\phi'_0(x)| >3$.
 By induction, $|\phi_j'(x)|\geq  3^{j}$ for $x \in d_{P,j}$. 
By the construction,
$\phi_\infty$ is a box mapping with long branches mapping over $U$ and  short branches, whose domains all lie in $d_{P,1}$, mapping over $Z$. 
The total length of the short branches is bounded by 
\begin{equation} \label{eqn:zeps}
    |Z| \sum_{j \geq 0} 3^{-j} \lesssim \sqrt\eps.
\end{equation}
Applying appropriate translations $T_j$, one can view the branch domains $d_{P,j}$ as pairwise disjoint and the collection of branches $\zeta_{P,j}$ as having uniform exponential tails; then apply Lemma~\ref{lemTails} (with $H = |\phi_0'|$ and $g = \{\zeta_{P,j} \circ T_j\}$) to deduce that
$\phi_\infty$ has uniform exponential tails.


\subparagraph{Pull-back by $\phi$.}
We transform the initial box mapping $\phi$ into ${\phi_*}$ by pull-back: 
$${\phi_*}(x)=\left\{
\begin{array}{ll}
h_V \circ \phi_\infty \circ \phi(x) & \text{ if } x\in Z,\\
h_V \circ \phi(x) & \text{ if }x \in V \setminus Z.
\end{array}
\right.$$
Thus defined, ${\phi_*}$
is  a box mapping. 
Recall that $h_V$ coincides with the identity map on $V$ and that all branches of $h_V$ are long, diffeomorphic onto $V$ and extensible over $U$. 
On $V\setminus Z$, branches of $h_V\circ \phi$ are long and extensible over $U$.
By construction of $\phi_\infty$, the domain of a long branch $\zeta$ of $\phi_\infty$ does not contain $\psi(0)$. Each branch of $h_V \circ \zeta$ is long, mapping onto $V$, and is extensible over $U$ with extension domain contained in the domain of $\zeta$ and, therefore, not containing $\psi(0)$. The short branches of $h_V \circ \phi_\infty$ coincide with those of $\phi_\infty$. 
Hence $\phi_*$ has two types of monotone branches,
short ones mapping over $Z$ and long ones mapping over
$V$, extensible over $U$.  The short branches and the possible central branch of $\phi_*$ have total length  
$$\lesssim~ \eps^{3/4}$$ by \eqref{eqn:zeps} and Fact~\ref{fa:fund} (iv).
Via Lemma~\ref{lemTails}, $h_V \circ \phi$ restricted to $V\setminus Z$ and $h_V \circ \phi_\infty$ on $U$ have uniform exponential tails. 
Consequently, using Fact~\ref{fa:fund} (iv) and Lemma~\ref{lemTailscrit}, $\phi_*$ has uniform exponential tails.

\subsection{Construction of Cantor repeller.}\label{sec:cantorconstr}
We finish the construction of the Cantor repeller $\varphi$
of Proposition~\ref{prop:ext} for a map $f=f_c$ with $c$ in the domain of Fact~\ref{fa:fund}. 
Because $f^n(\partial V) \cap U = \emptyset$ for all $n\geq 1$, no two branches of $\phi_*$ are adjacent. 

Let $\hat \phi_*$ denote $\phi_*$ with its short branches removed. 
Define $\tilde{\varphi}$ by retaining from the $K$th iterate $\hat \phi^K_*$ a finite number of branches contained in $V$ whose domains have union whose measure is at least $|V| - C_1 {\eps^{3/4}},$ where $K\geq 1$ is large enough to ensure that the map $\vp$ of the following paragraph satisfies $|\vp'| >2$, via the uniform distortion bound. Inherited from the same property for $\phi_*$, no two branches of $\tvp$ are adjacent.

Consider a branch $\zeta$ of $\tvp$ with domain $I_\zeta \subset \R$. Its inverse extends univalently to a map $\zeta^{-1}$ defined on 
 $\C\setminus (\R\setminus {U})$.
Let $D_\zeta:=\zeta^{-1}(\D_V)$, so $D_\zeta \cap \R = I_\zeta$. 
Because of negative Schwarzian, these inverse branches have the property of contracting Poincar\'e disks so $$D_\zeta \subset \D_{I_\zeta}$$
(see \cite[Fact~2.1.2]{book})   and thus their images $D_\zeta$ have pairwise-disjoint closures. We still need to add one imaginary branch. 

Let $g$ denote the restriction of $f$ to the right half-plane $\{\Re(z) > 0\}$, so $g$ is invertible. 
The sets $g^{-k}(\D_U)$ accumulate geometrically on the fixed point $p$, just as in \eqref{eqgeomacc}: 
$$
\dist{g^{-k}(\D_U), p} \sim \diam(g^{-k}(\D_U)) \sim \diam(g^{-k-2}(\D_U)) .$$
There exists $m \geq 1$ such that $f^2(0) \in \overline{g^{-m}(\D_U)}$; for this $m$,  $f^2(0) \notin g^{-k}(\D_U)$ for $k>m$. 
As $p-f^2(0) \sim \eps$, 
$$\eps \sim \diam(g^{-m-2}(\D_U))
\sim \diam(g^{-m-2}(\D_V))
.$$ 
Let $W$ be one of the two connected components of $f^{-1}(-g^{-m-2}(\D_V))$. As $f(W) \cap \R \subset (-p, c)$, $W \cap \R = \emptyset$.  
The diameter satisfies $\diam W \sim  \sqrt{\eps}.$
Set $\varphi = f^{m+3}$ on $W$.
Branches of $\tvp$ are contained in $f^{-1}(-\cup_{j=1}^{m+1}g^{-j}(\D_U))$ and so have domains whose closures are disjoint from $\overline{W}$.

The inverse of $\varphi_{|W}$ similarly extends univalently to 
 $\C\setminus (\R\setminus {U})$.
 Hence the distortion of iterates of $\vp$ is uniformly bounded independently of $\eps$. 

The map $\vp$ satisfies all claims of
Proposition~\ref{prop:ext}. 


\subsection{Thermodynamical formalism.}

Let $\varphi: {\calD} \mapsto \C$ be a Cantor repeller. Its Julia set
$${\J}_\varphi=
\bigcap_{n\geq 0} \varphi^{-n}(\C)$$
is a fully invariant Cantor set. It is well known that
$\varphi$ has an absolutely continuous invariant probabilistic measure $\sigma$ with respect to the $\HD({\J_{\vp}})$-conformal probabilistic  measure $\nu$, $\sigma$ is  an ergodic Gibbs measure with the H\"older potential ${\G}(x)=-\HD(\J_{\vp})\log |\vp'(x)|.$

The Hausdorff dimension of $\J_\vp$ is the unique solution of the equation $P_\vp(t)=0$, where the  pressure function $t\in\R\mapsto P(t)$ defined by
$$P_\vp(t)= \lim_{n\rightarrow \infty}\frac{1}{n}
\log \sum_{y \in {\varphi}^{-n}(0)} 
 |(\vp^n)'(y)|^{-t}$$
is an analytic and strictly convex function on $\R$, see~\cite{hans}.

The density $h(x)=\frac{d\sigma}{d\nu}(x)>0$ is a bounded measurable function and is a fixed point of the Perron-Frobenius operator for $\vp$ that acts on continuous functions $g:\J_\vp\mapsto \R$,
$${\mathcal L}_\vp(g)(x)=\sum_{y\in\vp^{-1}(x)}
g(y)\exp({\mathcal G}(y)),$$
and $h$ is the limit of 
\begin{equation*}\label{equ:h}
\frac{1}{n}\sum_{k=0}^{n-1}{\mathcal L}_\vp^k(1)(x)
\end{equation*}
in $L^1(\nu)$ topology. From the uniform distortion bound for branches of iterates of $\vp$ one obtains upper and positive lower bounds on $h$. The invariant measure $\sigma$
is a fixed point of the dual operator ${\mathcal L}_\vp^*(g)$
that acts on the space of Borel measures~\cite{ruellegaz}.


By the Birkhoff ergodic theorem, almost surely with respect to $\nu$,
\begin{equation}\label{equ:birk}
\lim_{n\rightarrow \infty}\frac{1}{n}\sum_{k=0}^{n-1}\log |\vp'(\vp^k(x))|=  \int_{\J_\vp} \log |\vp'(x)| d\sigma .
\end{equation}
\paragraph{Families of Cantor repellers.}
Let us return to our one-parameter family 
$\vp_\eps$, $\eps=c+2>0$, of Cantor repellers given by  Proposition~\ref{prop:ext} and its real counterpart family 
$\tvp_\eps$, $\eps>0$, defined by restricting $\vp_\eps$ to the real line. To simplify notation, we write $\tilde \J_\eps$ for $\J_{\tilde\vp_\eps}$, the Julia set of the real counterpart, and ${\calD}_n$ for the domain of $\tilde \vp^n_\eps$, with connected components denoted by $d$. 
Let 
$$Q(n,t) = 
\sum_{d\subset {\calD}_{n}} |d|^{t}.
$$
\begin{lem}\label{lem:start}
    There is a  uniform constant  $K>0$ such that, for any  $\rho \in (0,1)$, 
      the following holds for all $\eps >0$ small enough. 
\begin{itemize}
\item 
    $\HD(\tilde \J_\eps)\geq 1-K{\eps^{3/4}},$
\item
    $Q(n,1+\rho\sqrt{\eps}) \geq K^{-1}\left(1-{4K}\rho\sqrt{\eps}\right)^n$ for all $n \geq 1$. 
\end{itemize}
\end{lem}
\begin{proof}
    From \eqref{largeintersection} and bounded distortion, there exists a uniform $C'$  (note $|V| > 1$) such that, for $n\geq 1$,
$$\sum_{d\subset {\calD}_{n}} |d| \geq (1-C'\eps^{3/4})^n.$$

 If we put  $t=1-K{\eps^{3/4}}$ and $K > 2C'/\log 2$, then
\begin{eqnarray*}
    Q(n,t)
    ~& = &~ \sum_{d\subset {\calD}_{n}} |d|^{1-K{\eps^{3/4}}}\geq \frac{(1-C'\eps^{3/4})^n}{
\max_{d\subset {\calD}_{n}} |d|^{K{\eps^{3/4}}}}\\
 &\geq &
\frac{(1-C'\eps^{3/4} )^n}{2^{-K{\eps^{3/4}}n}} 
\geq (1+(K\log 2- 2C'){\eps^{3/4}})^n.
 \end{eqnarray*}
As 
 $$\sum_{y \in {\tilde{\varphi}_\eps}^{-n}(0)} 
 |(\vp_\eps^n)'(y)|^{-t} \sim Q(n,t),$$
 the pressure function ${P_{\tilde \vp_\eps}}(t) = \lim_{n\to \infty} \frac1n \log Q(n,t)$ is positive, which implies the dimension estimate. 

The map $\tvp_\eps$ has an invariant probabilistic  measure $\sigma_\eps$ with respect to the $\HD(\tilde \J_\eps)$-conformal measure $\nu_\eps$, both supported on $\R$. 
From the uniformly bounded distortion of iterates of $\vp_\eps$, there is a uniform bound on the densities with respect to each other. 
By bounded distortion and conformality of $\nu_\eps$, the Lyapunov exponent  
$$ \chi_\eps = \int \log |\tvp_\eps'| \, d\sigma_\eps ~\lesssim~ \sum_\zeta \inf|\zeta'|^{-\HD(\tilde \J_\eps)} \log \inf|\zeta'|, $$
where the sum is over branches of $\tvp_\eps$. 
Applying Corollary~\ref{corTails}, we deduce that  
    $$\chi_\eps < \sum_{n \geq 0} C'e^{(1-\alpha)n} e^{-n(1-K\eps^{3/4})}(n+1) \leq K',
$$ for   uniform constants $C', K'>0$. We redefine $K := \max(K,K')$. 

For the final estimate, we make use of the Chebyshev Inequality, 
$$
\sigma_\eps( \{x : \log |(\tvp_\eps^n)'(x)| \geq 2nK\}) 
\leq  \frac{n\chi_\eps}{2nK} < \frac{nK}{2nK} = \frac12.$$
Let $I_n$ denote the collection of connected components of 
${\calD}_{n}$ which contain at least one point of $ \{x : \log |(\vp_\eps^n)'(x)| < 2nK\}.$ Then 
$$
\sum_{d\in I_n}\nu_\eps\left( d \right) ~\grtsim~
\sum_{d\in I_n}\sigma_\eps\left( d \right) > \frac12. 
$$
By bounded distortion, for $d \in I_n$, $|d| ~\grtsim~ e^{-2nK}.$
Moreover, $$\sigma_\eps(d) \sim \nu_{\eps}(d) \sim |d|^{\HD(\tilde \J_\eps)}.$$
Then
\begin{eqnarray}
    Q(n, 1+\rho\sqrt{\eps})
    &\geq &\sum_{d \in I_n} |d|^{\HD(\tilde \J_{\eps})+\rho\sqrt\eps +K{\eps^{3/4}}} ~\grtsim~
\sum_{d \in I_n}{\nu_\eps(d)}~ |d|^{2\rho\sqrt{\eps}}\nonumber\\
~&\grtsim & 
~ \frac12 e^{-4K\rho\sqrt{\eps}n} \nonumber\\
& \geq &  \frac12 (1 - 4K \rho\sqrt\eps)^n .\nonumber
\end{eqnarray}
\end{proof}

For $t \in [0,2]$, by bounded distortion there is a constant $C>1$ such that 
\begin{equation} \label{eqnzqn}
\inf_{z\in \D_V} \sum_{y \in \tilde\vp^{-n}_\eps(z)} |(\tilde\vp_\eps^n)'(y)|^{-t} > C^{-1}Q(n,t).
\end{equation}
We now turn our attention to complex estimates. 
  \begin{prop}\label{prop:haus} 
      There is a universal constant $\rho>0$ such that, for all small $\eps >0$,  the Hausdorff dimension of the Julia set ${\J}_\varphi$ 
      of the Cantor repeller  $\varphi = \vp_\eps$  supplied by  Proposition~\ref{prop:ext} 
      satisfies
$$\HD({\J}_\varphi) \geq 1 +\rho\sqrt{\eps}.$$  
  \end{prop}
  \begin{proof}
Let us denote the complex branch of $\vp$ by $\zeta : W \to \D_V$. 

We wish to estimate  $\sum_{y \in \vp^{-n}(0)} |(\vp^n)'(y)|^{-t}.$
  We can decompose the set $\vp^{-n}(0)$ into a disjoint union of sets
  $$
  S_\alpha = 
  \zeta^{-j_k}\circ \tvp_\eps^{-j_{k-1}} \circ
  \cdots
  \circ \zeta^{-j_2}\circ \tvp_\eps^{-j_1} \circ  
  \zeta^{-j_0}(0)
  $$
  indexed by words $\alpha = j_0\ldots j_k$ for which $0 \leq j_0, j_k \leq n$, $1\leq j_1, \ldots, j_{k-1} \leq n$ 
  and $\sum_i{j_i} = n$. Note that, for small $\eps > 0$,  
  $$\eps^{\frac12 \rho\sqrt\eps} > \frac12.$$
  Then, with $t = 1+\rho\sqrt\eps$ and $\kappa = 4K\rho $, using estimates \eqref{eqnWest} from Proposition~\ref{prop:ext}, \eqref{eqnzqn} and Lemma~\ref{lem:start},
  \begin{align*}
      \sum_{y \in S_\alpha} |(\vp^n)'(y)|^{-t}
      & \geq (C^{-1}\sqrt{\eps})^{t(j_0 + j_2 + \cdots + j_k)} C^{-1} Q(j_1,t)
      \ldots C^{-1} Q(j_{k-1},t)  \\
      & \geq C^{-1}K^{-1} (2^{-1}C^{-2t}K^{-1}\sqrt{\eps})^{j_0 + j_2 + \cdots + j_k} (1-\kappa\sqrt{\eps})^n. 
  \end{align*}
  Each word $\alpha$ corresponds to a path in a binomial tree, starting by descending $j_0$ left branches, then $j_1$ right branches, then $j_2$ left branches and so on. 
  Summing over $\alpha$, we obtain (reversing the binomial expansion)
  $$\sum_{y \in \vp^{-n}(0)} |(\vp^n)'(y)|^{-t} \geq C^{-1}K^{-1}(1+2^{-1}C^{-3}K^{-1}\sqrt\eps)^n(1-\kappa\sqrt\eps)^n
  $$
  which is greater than $1$ for large $n$, if we choose $\rho < (8C^3K^2)^{-1}$.
  Hence the dimension is at least $t = 1+ \rho \sqrt\eps$ for all small $\eps >0$. 
  \end{proof}


  \begin{coro}
  Theorem \ref{theo:hausBot}\label{sec:low} holds.
  \end{coro}
  \begin{proof}
Simply note that ${\J}_c \supset \J_\vp$ so 
  $\HD({\J}_{c})\geq \HD({\J}_{\varphi}).$
\end{proof}

\section*{Declaration}
    The authors have no competing interests to declare that are relevant to the content of this article.

\section*{Acknowledgments}
The first and third authors benefited from a 2016 Research in Pairs stay at CIRM-Luminy, France.  
    Universit\'e Paris Est Cr\'eteil hosted the first author as  a Visiting Professor in 2018. We are very grateful to both institutions for supporting this research.

    We thank the referee for a careful reading of the paper and many helpful comments.

\end{document}